\documentclass[12pt]{amsart}
\usepackage{amsfonts, amssymb, latexsym, tikz-cd, algorithm, algorithmic, hyperref}

\newtheorem{theorem}{Theorem}[section]
\newtheorem{proposition}[theorem]{Proposition}
\newtheorem{lemma}[theorem]{Lemma}

\newtheorem{claim}[theorem]{Claim}
\newtheorem{subclaim}[theorem]{Subclaim}
\newtheorem*{claim*}{Claim}
\newtheorem{corollary}[theorem]{Corollary}
\newtheorem{Main Conjecture}[theorem]{Main Conjecture}
\newtheorem{conjecture}[theorem]{Conjecture}
\theoremstyle{remark}
\newtheorem{definition}[theorem]{Definition}
\newtheorem{example}[theorem]{Example}

\theoremstyle{plain}

\newcommand{\rn}[2]{
    \tikz[remember picture,baseline=(#1.base)]\node [inner sep=0] (#1) {$#2$};%
}

\newcommand\aug{\fboxsep=-\fboxrule\!\!\!\fbox{\strut}\!\!\!}

\newcommand{\cellsize}{12}
\newlength{\cellsz} 
\newsavebox{\cell}
\sbox{\cell}{\begin{picture}(\cellsize,\cellsize)
\put(0,0){\line(1,0){\cellsize}}
\put(0,0){\line(0,1){\cellsize}}
\put(\cellsize,0){\line(0,1){\cellsize}}
\put(0,\cellsize){\line(1,0){\cellsize}}
\end{picture}}
\newcommand\cellify[1]{\def\thearg{#1}\def\nothing{}%
\ifx\thearg\nothing
\vrule width0pt height\cellsz depth0pt\else
\hbox to 0pt{\usebox{\cell} \hss}\fi%
\vbox to \cellsz{
\vss
\hbox to \cellsz{\hss$#1$\hss}
\vss}}
\newcommand\tableau[1]{\vtop{\let\\\cr
\baselineskip -16000pt \lineskiplimit 16000pt \lineskip 0pt
\ialign{&\cellify{##}\cr#1\crcr}}}

\newcommand{\kellsize}{8}
\newlength{\kellsz} 
\newsavebox{\kell}
\sbox{\kell}{\begin{picture}(\kellsize,\kellsize)
\put(0,0){\line(1,0){\kellsize}}
\put(0,0){\line(0,1){\kellsize}}
\put(\kellsize,0){\line(0,1){\kellsize}}
\put(0,\kellsize){\line(1,0){\kellsize}}
\end{picture}}
\newcommand\kellify[1]{\def\thearg{#1}\def\nothing{}%
\ifx\thearg\nothing
\vrule width0pt height\kellsz depth0pt\else
\hbox to 0pt{\usebox{\kell} \hss}\fi%
\vbox to \kellsz{
\vss
\hbox to \kellsz{\hss$#1$\hss}
\vss}}
\newcommand\ktableau[1]{\vtop{\let\\\cr
\baselineskip -16000pt \lineskiplimit 16000pt \lineskip 0pt
\ialign{&\kellify{##}\cr#1\crcr}}}

\newcommand{\excise}[1]{}

\newcommand\concat{\mathbin{+\mkern-10mu+}}

\begin{document}
\pagestyle{plain}
\title{The Kostka semigroup and its Hilbert basis}
\author{Shiliang Gao}
\author{Joshua Kiers}
\address{Dept.~of Mathematics
Ohio State University,
Columbus, OH 43210, USA}
\email{kiers.2@osu.edu}
\author{Gidon Orelowitz}
\author{Alexander Yong}
\address{Dept.~of Mathematics, U.~Illinois at Urbana-Champaign, Urbana, IL 61801, USA} 
\email{sgao23@illinois.edu, gidono2@illinois.edu, ayong@illinois.edu}
\date{August 3, 2021}
\maketitle

\vspace{-.2in}
\begin{abstract}
The \emph{Kostka semigroup} consists of partitions $(\lambda,\mu)$ with at most $r$ parts that 
have positive Kostka coefficient. For this semigroup, Hilbert basis membership is an {\sf NP}-complete problem. We introduce  
\emph{KGR graphs} and \emph{conservative subtrees}, through the Gale-Ryser theorem on contingency tables, as a criterion for membership. Our main application shows
that if $(\lambda,\mu)$ is in the Hilbert basis then $\lambda_1\leq r$. 
We also classify the extremal rays of the associated polyhedral cone; these rays correspond to a (strict) subset of the Hilbert basis. 
In an appendix, the second and third authors show that a natural extension of our main result on the Kostka semigroup cannot be extended to the Littlewood-Richardson semigroup. This furthermore gives a counterexample to a recent speculation of P. Belkale concerning the semigroup controlling nonvanishing conformal blocks. 
\end{abstract}

\vspace{-.2in}
\tableofcontents
\vspace{-.5in}

\section{Introduction}

\subsection{Background}
Fix $r$ and $n$. Let $\lambda=(\lambda_1\geq \lambda_2\geq \ldots \geq \lambda_r)$ and $\mu=(\mu_1 \geq \mu_2 \geq \ldots \geq \mu_r)$ be integer partitions of $n$ with at most $r$ nonzero parts; let ${\sf Par}_r(n)$ be the set of such partitions. A \emph{semistandard tableau} of shape $\lambda$ and content $\mu$ is a filling of 
$\lambda$ (depicted as a Young diagram) with 
entries $1,2,\ldots,r$ such that:
\begin{itemize}
\item the rows weakly increase from left to right;
\item the columns strictly increase from top to bottom; and
\item there are $\mu_i$ many $i$'s.
\end{itemize}
The \emph{Kostka coefficient} $K_{\lambda,\mu}$ is the number of such tableaux. It is a basic notion in symmetric functions; see \emph{e.g.}, \cite{ECII}. 

\begin{example}
$K_{(4,2,1),(3,2,1,1)}=4$, as witnessed by the following semistandard tableaux:
\[\tableau{1 & 1 & 1 & 2\\ 2 &  3\\ 4},\ \  
\tableau{1 & 1 & 1 & 2\\ 2 &  4\\ 3}, \ \ 
 \tableau{1 & 1 & 1 & 3\\ 2 &  2\\ 4},\ \ 
 \tableau{1 & 1 & 1 & 4\\ 2 &  2\\ 3}.
\]
\end{example}

Define \emph{dominance order} on ${\sf Par}_r(n)$ by setting 
\begin{equation}
\label{eqn:dominance}
\lambda\geq_{\sf Dom} \mu \text{\ \ if  } \sum_{i=1}^t \lambda_i \geq \sum_{i=1}^t \mu_i, \text{ for  $1\leq t\leq r$.}
\end{equation}
It is textbook \cite[Proposition~7.10.5 and Exercise~7.12]{ECII} that 
\begin{equation}
\label{eqn:itistextbook}
K_{\lambda,\mu}> 0\iff \lambda\geq_{\sf Dom} \mu.
\end{equation}

Define the \emph{Kostka cone} 
\[{\sf Kostka}_r=\left\{(\lambda,\mu)\in {\mathbb R}^{2r}:
\begin{matrix}
\lambda_1\geq \lambda_2\geq \cdots \geq \lambda_r\geq 0\\
\mu_1\geq \mu_2\geq \cdots \geq \mu_r\geq 0\\
\sum_{i=1}^t \lambda_i \geq \sum_{i=1}^t \mu_i, \text{ for  $1\leq t\leq r-1$}\\
\sum_{i=1}^r \lambda_i = \sum_{i=1}^r \mu_i
\end{matrix}
\right\}\subseteq {\mathbb R}^{2r}.\]

The lattice points of ${\sf Kostka}_r$, namely,
\[{\sf Kostka}_r^{\mathbb Z}:={\sf Kostka}_r\cap {\mathbb Z}^{2r},\]
 form a semigroup. That is, 
 \[K_{\alpha,\beta}>0, K_{\lambda,\mu}>0\implies K_{\alpha+\lambda,\beta+\mu}> 0.\] From the general theory, ${\sf Kostka}_r^{\mathbb Z}$ has a unique minimal generating set as we now recall. 
 
 Let $C$ be any rational polyhedral cone inside $\mathbb{R}^m$. A finite set $\{a_1, \hdots, a_k\}\subset C\cap \mathbb{Z}^m$ is called a \emph{Hilbert basis} of $C$ if every integral point in $C$ can be realized as a nonnegative integral combination of $a_1, \hdots, a_k$. 
 Hilbert bases are known to exist for every rational polyhedral cone $C$; moreover, when $C$ is pointed, \emph{i.e.}, $C\cap -C = \{{\bf 0}\}$, there exists a unique minimal (with respect to set inclusion) Hilbert basis, see \cite[Theorem~16.4]{Schrijver}. We refer to that minimal Hilbert basis as \emph{the} Hilbert basis of $C$. 
 
 The uniqueness proof cited above characterizes the Hilbert basis $H$ of $C$ as follows: 
 $$
 H = \{a\in C\cap \mathbb{Z}^m | a\ne 0; a \text{ is not the sum of two nonzero elements $b,c\in C\cap \mathbb{Z}^m$}\}.
 $$
 In other words, the Hilbert basis consists of the irreducible elements of the semigroup $C\cap \mathbb{Z}^m$. 
 
Since ${\sf Kostka}_{r}$ is determined by rational inequalities, it is a rational polyhedral cone. Moreover, ${\sf Kostka}_r$ is pointed. Therefore, ${\sf Kostka}_r^{\mathbb Z}$ is finitely generated and has a (unique minimal) Hilbert basis.  See Table~\ref{table:Hilb4} for the case $r=4$. 

\begin{table}[t]
\centering
\begin{tabular}{ c c c c }
 $\ktableau{\ }, \ktableau{\ }$  & $\ktableau{ \ \\ \ }, \ktableau{ \ \\ \ }$ &
 $\ktableau{ \ \\ \ \\ \ }, \ktableau{ \ \\ \ \\ \ }$  & $\ktableau{ \ \\ \ \\ \ \\ \ }, \ktableau{ \ \\ \ \\ \  \\ \ }$ \\ \noalign{\vskip 2mm}    
 $\ktableau{\ & \ }, \ktableau{\ \\ \ }$  &  $\ktableau{\ & \  \\ \ }, \ktableau{\ \\ \  \\ \ }$ &
 $\ktableau{\ & \  \\ \ \\ \ }, \ktableau{\ \\ \  \\ \ \\ \ }$  &  $\ktableau{ \ & \ \\ \ & \ }, \ktableau{\ \\ \  \\ \ \\ \ }$ \\ \noalign{\vskip 2mm}    
 $\ktableau{ \ & \ \\ \ & \ } , \ktableau{ \ & \ \\ \ \\ \ }$  & $\ktableau{ \ & \ \\ \ & \ \\ \ } , \ktableau{ \ & \ \\ \ \\ \ \\ \ }$ &
 $\ktableau{ \ & \ \\ \ & \ \\ \ & \ } , \ktableau{\ & \ \\ \ & \ \\ \ \\ \ }$  & $\ktableau{\ & \ & \ }, \ktableau{\ \\ \ \\ \ }$ \\ \noalign{\vskip 2mm}    
 $\ktableau{\ & \ & \ \\ \ }, \ktableau{ \ \\ \ \\ \ \\ \ }$  &
 $\ktableau{ \ & \ & \ \\ \ & \ & \ } , \ktableau{ \ & \ \\ \ & \ \\ \ & \ }$ &
 $\ktableau{ \ & \ & \ \\ \ & \ & \ }, \ktableau{ \ & \ & \ \\ \ \\ \ \\ \ }$  &
 $\ktableau{ \ & \ & \  \\ \ & \ & \ \\ \ & \ }, \ktableau{\ & \ \\ \ & \ \\ \ & \ \\ \ & \ }$ \\ \noalign{\vskip 2mm}    
 $\ktableau{\ & \ & \ \\ \ & \ & \ \\ \ & \ & \ } , \ktableau{ \ & \ & \ \\ \ & \ \\ \ & \ \\ \ & \ }$  &
 $\ktableau{ \ & \ & \ & \ } , \ktableau{ \ \\ \ \\ \ \\ \ }$ &
 $\ktableau{ \ & \ & \ & \ \\  \ & \ & \ & \ \\ \ & \ & \ & \ }, \ktableau{ \ & \ & \ \\ \ & \ & \ \\ \ & \ & \ \\ \ & \ & \ }$  & \   
 \end{tabular}
 \vspace{2mm}    
 \caption{The Hilbert basis for ${\sf Kostka}_{4}^\mathbb{Z}$.} \label{table:Hilb4}
\end{table}

The general problem, of deciding if a lattice point in a pointed cone
is in the Hilbert basis, is {\sf NP}-complete \cite[Theorem~3.1]{Henk}. This formal
indication of difficulty \emph{specifically} holds for the Kostka cone (with a similar argument). Define a decision problem {\tt KostkaHilbert} as follows. The input is $(\lambda,\mu)\in {\sf Kostka}_r^{\mathbb Z}$ described by their \emph{columns} (\emph{i.e.}, we input the conjugate shapes $(\lambda',\mu')$). The output is whether $(\lambda,\mu)$ is
\emph{reducible}, that is, if there exist nontrivial
$(\lambda^{\bullet},\mu^{\bullet}),(\lambda^{\circ},\mu^{\circ})\in {\sf Kostka}_r^{\mathbb Z}$  such that
\begin{equation}
\label{eqn:decomposition}
(\lambda,\mu)=(\lambda^{\bullet},\mu^{\bullet})+(\lambda^{\circ},\mu^{\circ}).
\end{equation}
$(\lambda,\mu)\in {\sf Kostka}_r^{\mathbb Z}$ is \emph{irreducible} if it is not reducible. 
Measure the input $(\lambda,\mu)$ using bits and assume
arithmetic operations take constant time. This is proved in Appendix~\ref{sec:2}:

\begin{theorem}[{\emph{cf.} \cite[Theorem~3.1]{Henk}}]
\label{thm:complexity}
{\tt KostkaHilbert} is {\sf NP}-complete.
\end{theorem}

Hence, no polynomial-time algorithm exists for {\tt KostkaHilbert} unless
${\sf P}={\sf NP}$. This seems to rule out any simple, explicit classification of the Hilbert basis of ${\sf Kostka}_r^{\mathbb Z}$.

\subsection{Results} 
To each 
$(\lambda,\mu)\in {\sf Kostka}_r^{\mathbb Z}$ we introduce the \emph{Kostka-Gale-Ryser (KGR) graph} $G(\lambda,\mu)$ and the 
concept of  a \emph{conservative subtree}\footnote{Akin to ``conservative vector field'' and path independence.} of $G(\lambda,\mu)$.

\begin{theorem}[see~Theorem~\ref{thm:bigone}]\label{thm:mainbigone}
Let $(\lambda,\mu)\in {\sf Kostka}_r^{\mathbb Z}$.
If $G(\lambda,\mu)$ has a conservative subtree, then $(\lambda,\mu)$ is not a Hilbert basis element of 
${\sf Kostka}_r^{\mathbb Z}$.
\end{theorem}

The converse is false (see Theorem~\ref{thm:someifftrue} and discussion thereafter about 
Example~\ref{exa:conversenottrue}).
We apply Theorem~\ref{thm:mainbigone} to prove a
bound on the first coordinate of a Hilbert basis element. This is our main result:

\begin{theorem}[Width Bound]
\label{thm:bound}
Suppose $(\lambda,\mu)$ is a Hilbert basis element of ${\sf Kostka}_{r}^{\mathbb Z}$. Then $\lambda_1\le r$. Moreover, if
$\lambda_1=r$ then $\lambda$ and $\mu$ are both rectangles.
\end{theorem}
Theorem~\ref{thm:bound}'s bound is sharp since $((t),(1^t))\in {\sf Kostka}_r^{\mathbb Z}$ is in the basis, for $1\leq t\leq r$. The second
sentence implies $\mu_1\leq r-1$; this bound is also sharp since $((r-1) \times r ,r\times (r-1))$ is in the basis. Let $\ell(\lambda)$ be the number of nonzero parts of $\lambda$.
\begin{corollary}
Suppose $\lambda$ and $\mu$ are partitions of the same size, and $\lambda\geq_{\sf Dom} \mu$. If $\ell(\mu)<\lambda_1$
then $(\lambda,\mu)$ is not in the Hilbert basis of ${\sf Kostka}_r^{\mathbb Z}$ for any $r\geq \ell(\mu)$. 
\end{corollary}
\begin{proof}
The Hilbert basis of ${\sf Kostka}_{r'}^{\mathbb Z}$ includes into that of ${\sf Kostka}_r^{\mathbb Z}$ whenever $r'\leq r$.
\end{proof}

An \emph{extremal ray} is a one-dimensional face of the cone ${\sf Kostka}_r$. Every element along such a ray is not the sum of two nonparallel elements of ${\sf Kostka}_r$; hence the first lattice point along such a ray is not the sum of two nontrivial elements of ${\sf Kostka}_r^{\mathbb{Z}}$. Thus, we have an injection from the set of extremal rays to the Hilbert basis by taking this first lattice point. 
We classify the extremal rays of ${\sf Kostka}_r$, thereby describing this subset of the Hilbert basis. 

\begin{theorem}
\label{thm:rays}
Suppose $(\lambda,\mu)\in {\sf Kostka}_{r}^{\mathbb Z}$. $(\lambda,\mu)$ lies on an extremal ray of ${\sf Kostka}_{r}$ if and only if it is a positive rational multiple of one of the following: 
$$
\left(\underbrace{a,\hdots,a}_{b+\ell},0,\hdots,0; \underbrace{a,\hdots,a}_{\ell},\underbrace{b,\hdots,b}_{a},0\hdots,0\right)
$$
where $a,b,\ell$ are integers such that $r\ge a+\ell\ge a\ge b>0$. 
\end{theorem}

Counting possibilities for $b$, $a$, and $a+\ell$ above immediately gives: 

\begin{corollary}
The number of extremal rays of ${\sf Kostka}_r$ is $\binom{r}{3}+\binom{r}{2}+\binom{r}{1}$.
\end{corollary}

In contrast, see Table \ref{table:growth}.  For any $s\le r$ and $|\lambda| = s$, $(\lambda,1^s)$ is in the Hilbert basis. So the number of Hilbert basis elements of ${\sf Kostka}_r^\mathbb{Z}$ is bounded below by the sum of partition functions (\url{http://oeis.org/A000070}).
What is the actual asymptotic growth?

\begin{table}[h]
\centering
\begin{tabular}{ c|c|c||c|c|c }
$r$ & $\#$ extremal rays & $\#$ Hilbert basis & $r$ & $\#$ extremal rays & $\#$ Hilbert basis \\\hline
$1$ & $1$ & $1$  & $10$ & $175$ & $3093$ \\
$2$ & $3$ & $3$ & $11$ & $231$ & $6876$ \\
$3$ & $7$ & $8$ & $12$ & $298$ & $14133$ \\
$4$ & $14$ & $19$ & $13$ & $377$ & $29788$ \\
$5$ & $25$ & $50$ & $14$ & $469$ & $59935$ \\
$6$ & $41$ & $111$ & $15$ & $575$ & $118893$ \\
$7$ & $63$ & $281$ & $16$ & $696$ & $232972$ \\
$8$ & $92$ & $635$ & $17$ & $833$ & $457982$  \\
$9$ & $129$ & $1443$ & & & 
\end{tabular}
\caption{Counting extremal rays and Hilbert basis elements of ${\sf Kostka}_{r}^\mathbb{Z}$.}\label{table:growth}
\end{table}

Kostka coefficients appear in the representation theory of $G = SL_r$. For each partition $\lambda$ with at most $r$ parts
 there is a finite-dimensional representation $V_\lambda$ of $G$, and $K_{\lambda,\mu}$ records the dimension of the largest subspace of $V_\lambda$ on which a fixed maximal torus $T\subset G$ acts with character described by $\mu$. For any reductive Lie group $G$, one 
defines the weight space dimensions $K_{\lambda,\mu}^G$ and one may study the cone of solutions $(\lambda,\mu)$ to the problem $K_{\lambda,\mu}^G > 0$. During preparation of this article, M.~Besson, S.~Jeralds, and the second author \cite{Kiers} proved a generalization of
Theorem~\ref{thm:rays}. Those methods involve analysis of dominant weight polytopes in arbitrary Lie type. Our argument uses no Lie theory.

\subsection{Organization} 
Section~\ref{sec:3} introduces the key constructions of this paper. There we use the \emph{Gale-Ryser theorem}, and specifically, the \emph{canonical matrix} $A(\lambda,\mu)$ which is a
$\{0,1\}$-matrix with prescribed row and column marginals. We define the auxiliary matrix $A^*(\lambda,\mu)$
from $A(\lambda,\mu)$. This permits us to define the \emph{KGR graph} $G(\lambda,\mu)$ and \emph{conservative
subtrees}. Existence of such a subtree rules out $(\lambda,\mu)$ being a Hilbert basis element; this is Theorem~\ref{thm:mainbigone} (see Theorem~\ref{thm:bigone}).  Our application to the proof of
the Width Bound (Theorem~\ref{thm:bound}) is in Section~\ref{sec:2.2}. In that section we also prove 
that existence of a conservative subtree
is equivalent to $A(\lambda,\mu)$ being reducible (Theorem~\ref{thm:someifftrue}). 
Section~\ref{sec:4} proves our result on extremal rays of 
${\sf Kostka}_{r}$ (Theorem~\ref{thm:rays}). Section~\ref{sec:5} discusses some conjectures that were posed
in an earlier version of this preprint. All of these conjectures have since been resolved (two of which by J.~Kim \cite{K:21}) and one, which concerns an extension to Littlewood-Richardson coefficients, has been proved false 
(Theorem~\ref{thm:counterex}). Appendix~\ref{sec:2} proves Theorem~\ref{thm:complexity}, by reducing the \emph{Subset Sum problem} to {\tt KostkaHilbert}.

\section{The Gale-Ryser theorem, contingency tables, and directed graphs}\label{sec:3}

\subsection{KGR graphs, and conservative subtrees} \label{sec:3.1}
Let ${\sf Par}_t$ be the set of partitions with at most $t$ nonzero parts.
Suppose $\alpha\in {\sf Par}_r$ and $\beta\in {\sf Par}_s$. Let ${\sf GR}(\alpha,\beta)$ be the set of $\{0,1\}$-matrices (contingency tables) $(A_{ij})_{1\leq i\leq r, 1\leq j\leq s}$ 
of dimension $r\times s$ such that 
\[\sum_{j=1}^s a_{ij} = \alpha_i, \ 1\leq i\leq r; \ \ \ \ \sum_{i=1}^r a_{ij} = \beta_j, \ 1\leq j\leq s.\]

The \emph{Gale-Ryser theorem} states that
\begin{equation}
\label{eqn:Gale-Ryser}
{\sf GR}(\alpha,\beta)\neq\emptyset \iff \beta\leq_{\sf Dom} \alpha'. 
\end{equation}
One reference is R.~Brualdi's \cite{Brualdi:older}, who recounts that (\ref{eqn:Gale-Ryser})
was independently obtained in 1957, by D.~Gale \cite{Gale} and H.~J.~Ryser \cite{Ryser}. The former used ``his supply-demand
theorem for network flows'' while the latter constructs a matrix in ${\sf GR}(\alpha,\beta)$, as described below.

\begin{proposition}
\label{thm:0-1}
$(\lambda,\mu)\in \sf{Kostka}_r^{\mathbb Z}$ if and only if ${\sf GR}(\mu,\lambda')\neq \emptyset$.
\end{proposition}
\begin{proof}
By (\ref{eqn:itistextbook}), and that $\leq_{\sf Dom}$ is an anti-automorphism under conjugation \cite[Section~7.2]{ECII},
\[(\lambda,\mu)\in \sf{Kostka}^{\mathbb Z}_r \iff \mu\leq_{\sf Dom} \lambda \iff \lambda'\leq_{\sf Dom} \mu'.\]
Hence by (\ref{eqn:Gale-Ryser}), this is equivalent to ${\sf GR}(\mu,\lambda')\neq \emptyset$.
\end{proof}

We also follow R.~Brualdi \cite{Brualdi}. Ryser gave an algorithm that takes $(\lambda,\mu)\in \sf{Kostka}_r^{\mathbb Z}$ as an input, and 
outputs a $\{0,1\}$-matrix $A(\lambda,\mu)$ of dimension $r\times \lambda_1$ that exhibits ${\sf GR}(\mu,\lambda')\neq \emptyset$, as guaranteed
by Proposition~\ref{thm:0-1}. Call this matrix the \emph{canonical matrix} for $(\lambda,\mu)$.  

We need to recall Ryser's algorithm. First define $A(\mu)$ to be the $r\times \lambda_1$
dimension $\{0,1\}$-matrix whose $i$-th row has $\mu_i$ many $1$'s placed flush-left. Let $\beta=\lambda'$ and $s=\lambda_1$.
Shift $\beta_s$ rightmost $1$'s in each row of $A(\mu)$ to column $s$, by choosing $1$'s in rows with the largest sum, and breaking
ties by choosing $1$'s in rows further south. Let $A'(\mu)$ be the submatrix of $A(\mu)$ using the first $s-1$ columns. Repeat the
shifting process, now to column $s-1$ so that there are $\beta_{s-1}$ many $1$'s there. Continuing, Ryser's algorithm eventually
arrives at $A(\lambda,\mu)$. We refer the reader to \cite{Brualdi:older} and the original \cite{Ryser} for additional details.

\begin{example}
\label{exa:Jan8aaa}
If $\lambda=(8,7,7,7,3,2)$ and $\mu = (7,7,4,4,4,4,4)$, then
\begin{equation*}
A(\mu) = \begin{bmatrix}
1 & 1 & 1 & 1 & 1 & 1 & 1 & 0 \\
1 & 1 & 1 & 1 & 1 & 1 & 1 & 0 \\
1 & 1 & 1 & 1 & 0 & 0 & 0 & 0 \\
1 & 1 & 1 & 1 & 0 & 0 & 0 & 0 \\
1 & 1 & 1 & 1 & 0 & 0 & 0 & 0 \\
1 & 1 & 1 & 1 & 0 & 0 & 0 & 0 \\
1 & 1 & 1 & 1 & 0 & 0 & 0 & 0 
\end{bmatrix}, A(\lambda,\mu) = \begin{bmatrix}
1 & 1 & 1 & 1 & 1 & 1 & 1 & 0 \\
1 & 1 & 1 & 0 & 1 & 1 & 1 & 1 \\
1 & 1 & 1 & 0 & 1 & 0 & 0 & 0 \\
1 & 1 & 0 & 1 & 0 & 1 & 0 & 0 \\
1 & 1 & 0 & 1 & 0 & 1 & 0 & 0 \\
1 & 0 & 1 & 1 & 0 & 0 & 1 & 0 \\
0 & 1 & 1 & 0 & 1 & 0 & 1 & 0 
\end{bmatrix}.
\end{equation*}
\end{example} 
Here $\lambda'=(6,6,5,4,4,4,4,1)$, and the steps of Ryser's algorithm are
\begin{equation*}
A(\mu)\mapsto 
\begin{bmatrix}
1 & 1 & 1 & 1 & 1 & 1 & 1 & \aug& 0 \\
1 & 1 & 1 & 1 & 1 & 1 & 0 & \aug & 1 \\
1 & 1 & 1 & 1 & 0 & 0 & 0 &\aug & 0 \\
1 & 1 & 1 & 1 & 0 & 0 & 0 &\aug & 0 \\
1 & 1 & 1 & 1 & 0 & 0 & 0 &\aug & 0 \\
1 & 1 & 1 & 1 & 0 & 0 & 0 &\aug & 0 \\
1 & 1 & 1 & 1 & 0 & 0 & 0 &\aug & 0 
\end{bmatrix}\mapsto
\begin{bmatrix}
1 & 1 & 1 & 1 & 1 & 1 & \aug & 1 & 0 \\
1 & 1 & 1 & 1 & 1 & 0 & \aug & 1 & 1 \\
1 & 1 & 1 & 1 & 0 & 0 & \aug & 0 & 0 \\
1 & 1 & 1 & 1 & 0 & 0 & \aug & 0 & 0 \\
1 & 1 & 1 & 1 & 0 & 0 & \aug & 0 & 0 \\
1 & 1 & 1 & 0 & 0 & 0 & \aug & 1 & 0 \\
1 & 1 & 1 & 0 & 0 & 0 & \aug & 1 & 0 
\end{bmatrix}
\mapsto
\begin{bmatrix}
1 & 1 & 1 & 1 & 1 & \aug & 1 & 1 & 0 \\
1 & 1 & 1 & 1 & 0 & \aug & 1 &  1 & 1 \\
1 & 1 & 1 & 1 & 0 & \aug & 0 &  0 & 0 \\
1 & 1 & 1 & 0 & 0 & \aug & 1 &  0 & 0 \\
1 & 1 & 1 & 0 & 0 & \aug & 1 &  0 & 0 \\
1 & 1 & 1 & 0 & 0 & \aug & 0 & 1 & 0 \\
1 & 1 & 1 & 0 & 0 & \aug & 0 &  1 & 0 
\end{bmatrix}
\end{equation*}
\begin{equation*}
\ \ \ \ \ \ \ \ \ \ \ \ \mapsto 
\begin{bmatrix}
1 & 1 & 1 & 1 & \aug & 1  & 1 & 1 & 0 \\
1 & 1 & 1 & 0 & \aug & 1 & 1 &  1 & 1 \\
1 & 1 & 1 & 0 & \aug & 1  & 0 &  0 & 0 \\
1 & 1 & 1 & 0 & \aug & 0  & 1 &  0 & 0 \\
1 & 1 & 1 & 0 & \aug & 0  & 1 &  0 & 0 \\
1 & 1 & 1 & 0 & \aug & 0  & 0 & 1 & 0 \\
1 & 1 & 0 & 0 & \aug & 1 & 0 &  1 & 0 
\end{bmatrix}
\mapsto 
\begin{bmatrix}
1 & 1 & 1 & \aug & 1 & 1  & 1 & 1 & 0 \\
1 & 1 & 1 & \aug & 0 & 1 & 1 &  1 & 1 \\
1 & 1 & 1 & \aug & 0 & 1  & 0 &  0 & 0 \\
1 & 1 & 0 & \aug & 1 & 0  & 1 &  0 & 0 \\
1 & 1 & 0 & \aug & 1 & 0  & 1 &  0 & 0 \\
1 & 1 & 0 & \aug & 1 & 0  & 0 & 1 & 0 \\
1 & 1 & 0 & \aug & 0 & 1 & 0 &  1 & 0 
\end{bmatrix}
\mapsto 
\begin{bmatrix}
1 & 1 & \aug & 1 & 1 & 1  & 1 & 1 & 0 \\
1 & 1 & \aug & 1 & 0 & 1 & 1 &  1 & 1 \\
1 & 1 & \aug & 1 & 0 & 1  & 0 &  0 & 0 \\
1 & 1 & \aug & 0 & 1 & 0  & 1 &  0 & 0 \\
1 & 1 & \aug & 0 & 1 & 0  & 1 &  0 & 0 \\
1 & 0 & \aug & 1 & 1 & 0  & 0 & 1 & 0 \\
1 & 0 & \aug & 1 & 0 & 1 & 0 &  1 & 0 
\end{bmatrix}
\end{equation*}
\begin{equation*}
\ \ \ \ \ \ \ \ \ \ \ \ \ \mapsto 
\begin{bmatrix}
1 & \aug & 1 & 1 & 1 & 1  & 1 & 1 & 0 \\
1 & \aug & 1 & 1 & 0 & 1 & 1 &  1 & 1 \\
1 & \aug & 1 & 1 & 0 & 1  & 0 &  0 & 0 \\
1 & \aug & 1 & 0 & 1 & 0  & 1 &  0 & 0 \\
1 & \aug & 1 & 0 & 1 & 0  & 1 &  0 & 0 \\
1 & \aug & 0 & 1 & 1 & 0  & 0 & 1 & 0 \\
0 & \aug & 1 & 1 & 0 & 1 & 0 &  1 & 0 
\end{bmatrix}\mapsto A(\lambda,\mu).
\end{equation*}

If $\mu$ is a partition then $\mu_i-\mu_{i+1}\geq 0$ for all $i$.  Define the column vector $\mu^*\in {\mathbb Z}_{\geq 0}^r$ by 
\[\mu^*_i = \mu_i - \mu_{i+1},\] 
where $\mu_{r+1}:=0$.  Similarly, define $A^*(\lambda,\mu)$ to be the matrix with the same dimensions as $A(\lambda,\mu)$, with 
\begin{equation}
\label{eqn:Jan12aaa}
A^*(\lambda,\mu)_{i,j} = A(\lambda,\mu)_{i,j} - A(\lambda,\mu)_{i+1,j},
\end{equation}
where $A(\lambda,\mu)_{r+1,j}:= 0$.  

\begin{example}
Given the same choices for $\lambda$ and $\mu$ as before, \begin{equation*}
A^*(\lambda,\mu) = \begin{bmatrix}
0 & 0 & 0 & 1 & 0 & 0 & 0 &-1 \\
0 & 0 & 0 & 0 & 0 & 1 & 1 & 1 \\
0 & 0 & 1 &-1 & 1 &-1 & 0 & 0 \\
0 & 0 & 0 & 0 & 0 & 0 & 0 & 0 \\
0 & 1 &-1 & 0 & 0 & 1 &-1 & 0 \\
1 &-1 & 0 & 1 &-1 & 0 & 0 & 0 \\
0 & 1 & 1 & 0 & 1 & 0 & 1 & 0 
\end{bmatrix}
\end{equation*}
\end{example}

\begin{definition}
The \emph{Kostka-Gale-Ryser (KGR) graph} $G(\lambda,\mu)$ 
is a directed graph obtained from $A^*(\lambda,\mu)$ by drawing an arc from each
\begin{itemize}
\item[(I)] $-1$ to the rightmost $1$ to the left of it in the same row 
\item[(II)] $1$ to the $-1$ in the same column (if it exists).  
\end{itemize}
\end{definition}
The well-definedness of (I) is established by Lemma~\ref{alternation}.

\begin{example}
Below $G(\lambda,\mu)$ is drawn in blue, overlaid on top of $A^*(\lambda,\mu)$.
\begin{equation*}
A^*(\lambda,\mu) = \begin{bmatrix}
\rn{11}{0} & \rn{12}{0} & \rn{13}{0} & \rn{14}{1} & \rn{15}{0} & \rn{16}{0} & \rn{17}{0} &\rn{18}{-1} \\
\rn{21}{0} & \rn{22}{0} & \rn{23}{0} & \rn{24}{0} & \rn{25}{0} & \rn{26}{1} & \rn{27}{1} & \rn{28}{1} \\
\rn{31}{0} & \rn{32}{0} & \rn{33}{1} &\rn{34}{-1} & \rn{35}{1} &\rn{36}{-1} & \rn{37}{0} & \rn{38}{0} \\
\rn{41}{0} & \rn{42}{0} & \rn{43}{0} & \rn{44}{0} & \rn{45}{0} & \rn{46}{0} & \rn{47}{0} & \rn{48}{0} \\
\rn{51}{0} & \rn{52}{1} &\rn{53}{-1} & \rn{54}{0} & \rn{55}{0} & \rn{56}{1} &\rn{57}{-1} & \rn{58}{0} \\
\rn{61}{1} &\rn{62}{-1} & \rn{63}{0} & \rn{64}{1} &\rn{65}{-1} & \rn{66}{0} & \rn{67}{0} & \rn{68}{0} \\
\rn{71}{0} & \rn{72}{1} & \rn{73}{1} & \rn{74}{0} & \rn{75}{1} & \rn{76}{0} & \rn{77}{1} & \rn{78}{0} 
\end{bmatrix}
\end{equation*}

\begin{tikzpicture}[overlay,remember picture]
        \draw [->, color=blue, line width=0.5mm] (18) -- (14);
        \draw [->, color=blue, line width=0.5mm] (14) -- (34);
         \draw [->, color=blue, line width=0.5mm] (34) -- (33);
          \draw [->, color=blue, line width=0.5mm] (33) -- (53);
           \draw [->, color=blue, line width=0.5mm] (53) -- (52);
 \draw [->, color=blue, line width=0.5mm] (52) -- (62);
 \draw [->, color=blue, line width=0.5mm] (62) -- (61);
 
 \draw [->, color=blue, line width=0.5mm] (27) -- (57);
  \draw [->, color=blue, line width=0.5mm] (57) -- (56);
   \draw [->, color=blue, line width=0.5mm] (26) -- (36);
     \draw [->, color=blue, line width=0.5mm] (36) -- (35);
      \draw [->, color=blue, line width=0.5mm] (35) -- (65);
      \draw [->, color=blue, line width=0.5mm] (65) -- (64); 
 \draw [->, color=blue, line width=0.5mm] (28) -- (18);
   \draw [->, color=blue, line width=0.5mm] (64) -- (34);
   \draw [->, color=blue, line width=0.5mm] (64) -- (34);
\draw [->, color=blue, line width=0.5mm] (56) -- (36);

\draw [->, color=blue, line width=0.5mm] (72) -- (62);
\draw [->, color=blue, line width=0.5mm] (73) -- (53);
\draw [->, color=blue, line width=0.5mm] (75) -- (65);
\draw [->, color=blue, line width=0.5mm] (77) -- (57);
    \end{tikzpicture}

\vspace{-.2in}
\end{example}
Define a \emph{source} of a directed graph to be a vertex $y$ with no arcs of the form $x\to y$. Similarly, a \emph{sink}
is a vertex $x$ without arcs $x\to y$. Hence $x$ may be both a source and sink if (and only if) it is an isolated vertex. In addition, \emph{connected}
will mean with respect to the underlying \emph{undirected} graph structure. A \emph{subtree} of a directed graph means
a connected subgraph with no undirected cycles.
 
\begin{definition}
A connected nontrivial subgraph $G'\subset G(\lambda,\mu)$ is a \emph{conservative subtree} if whenever a vertical arc 
appears in $G'$, all vertical arcs of $G(\lambda,\mu)$ in the same column do, and if one of the following is true:
\begin{enumerate}
\item[(C.1)]
$G'$ is a connected component of $G(\lambda,\mu)$ (with respect to the underlying undirected graph structure).
\item[(C.2)]
The unique sink of $G'$ is a $-1$ in $A^*(\lambda,\mu)$, one source of $G'$ is a $1$ in $A^*(\lambda,\mu)$ in the same row as the sink, and all other sources of $G'$ are sources of $G(\lambda,\mu)$.
\end{enumerate}
\end{definition}

\begin{example}
We have marked a conservative subtree of $G(\lambda,\mu)$ in red.
\begin{equation*}
\begin{bmatrix}
\rn{11}{0} & \rn{12}{0} & \rn{13}{0} & \rn{14}{1} & \rn{15}{0} & \rn{16}{0} & \rn{17}{0} &\rn{18}{-1} \\
\rn{21}{0} & \rn{22}{0} & \rn{23}{0} & \rn{24}{0} & \rn{25}{0} & \rn{26}{1} & \rn{27}{1} & \rn{28}{1} \\
\rn{31}{0} & \rn{32}{0} & \rn{33}{1} &\rn{34}{-1} & \rn{35}{1} &\rn{36}{-1} & \rn{37}{0} & \rn{38}{0} \\
\rn{41}{0} & \rn{42}{0} & \rn{43}{0} & \rn{44}{0} & \rn{45}{0} & \rn{46}{0} & \rn{47}{0} & \rn{48}{0} \\
\rn{51}{0} & \rn{52}{1} &\rn{53}{-1} & \rn{54}{0} & \rn{55}{0} & \rn{56}{1} &\rn{57}{-1} & \rn{58}{0} \\
\rn{61}{1} &\rn{62}{-1} & \rn{63}{0} & \rn{64}{1} &\rn{65}{-1} & \rn{66}{0} & \rn{67}{0} & \rn{68}{0} \\
\rn{71}{0} & \rn{72}{1} & \rn{73}{1} & \rn{74}{0} & \rn{75}{1} & \rn{76}{0} & \rn{77}{1} & \rn{78}{0} 
\end{bmatrix}
\end{equation*}

\begin{tikzpicture}[overlay,remember picture]
        \draw [->, color=red, line width=0.7mm] (18) -- (14);
        \draw [->, color=red, line width=0.7mm] (14) -- (34);
         \draw [->, color=red, line width=0.7mm] (34) -- (33);
          \draw [->, color=red, line width=0.7mm] (33) -- (53);
           \draw [->, color=red, line width=0.7mm] (53) -- (52);
 \draw [->, color=red, line width=0.7mm] (52) -- (62);
 \draw [->, color=blue, line width=0.5mm] (62) -- (61);
 
 \draw [->, color=blue, line width=0.5mm] (27) -- (57);
  \draw [->, color=blue, line width=0.5mm] (57) -- (56);
   \draw [->, color=blue, line width=0.5mm] (26) -- (36);
     \draw [->, color=blue, line width=0.5mm] (36) -- (35);
      \draw [->, color=blue, line width=0.5mm] (35) -- (65);
      \draw [->, color=blue, line width=0.5mm] (65) -- (64); 
 \draw [->, color=red, line width=0.7mm] (28) -- (18);
   \draw [->, color=red, line width=0.7mm] (64) -- (34);
\draw [->, color=blue, line width=0.5mm] (56) -- (36);

\draw [->, color=red, line width=0.7mm] (72) -- (62);
\draw [->, color=red, line width=0.7mm] (73) -- (53);
\draw [->, color=blue, line width=0.5mm] (75) -- (65);
\draw [->, color=blue, line width=0.5mm] (77) -- (57);
    \end{tikzpicture}
\end{example}
\vspace{-.2in}

This is our main result about KGR graphs and conservative subtrees:
\begin{theorem}
\label{thm:bigone}
Let $(\lambda,\mu)\in {\sf Kostka}_r^{\mathbb Z}$. 

The undirected graph structure of
$G(\lambda,\mu)$ is a forest that is planar (with respect to the specified embedding).
Therefore, any  conservative subtree of $G(\lambda,\mu)$ is in fact a tree.

If $G(\lambda,\mu)$ has a conservative subtree, then $(\lambda,\mu)$ is not in the Hilbert basis of
${\sf Kostka}_{r}^{\mathbb Z}$.
\end{theorem}

The proof, given in Section~\ref{sec:3.3}, proceeds by building 
a number of additional properties of the graphs and $A^*(\lambda,\mu)$, which we also use later.

\subsection{Proof of the main result on KGR graphs (Theorem~\ref{thm:bigone})} \label{sec:3.3}
Let 
\[A^{(0)}=A(\mu), A^{(1)}, A^{(2)},\ldots, A^{(\lambda_1)}=A(\lambda,\mu)\] 
be the sequence of matrices
in Ryser's algorithm. 

\begin{lemma}
\label{lemma:eachAcolumn}
Each column of $A(\lambda,\mu)$ is one of the following:
\begin{enumerate}
\item
a single run of $1$'s, starting in the top row;
\item
a single run of $1$'s, not starting in the top row; or
\item
two runs of $1$'s, one starting in the top row.
\end{enumerate}
Furthermore, the leftmost column must fall into case $(1)$.
\end{lemma}
\begin{proof}
We prove that the stated properties hold for every column of $A^{(i)}$. Clearly this is true for $i=0$. Now assume $i>0$ and the assertion holds for $A^{(i-1)}$. By construction, 
columns $\lambda_1-i+2,\lambda_1-i+3,\ldots,\lambda_1$ are the same as those of $A^{(i-1)}$. Hence by induction,
those columns satisfy the claim. In $A^{(i-1)}$ the $1$'s in columns $1,2,\ldots, \lambda_1-i$ form a Young diagram $\theta$.
Consider the region $R$ in those columns and strictly south of row $\lambda'_{\lambda_1-i+1}$; the $1$'s there also
form a sub-Young diagram $\underline\theta$ of $\theta$.  Ryser's rule $A^{(i-1)}\mapsto A^{(i)}$ is defined precisely to
ensure that the $1$'s removed from $\underline\theta$ leave a partition. Hence the claim holds for columns $1,2,\ldots,\lambda_1-i$
of $A^{(i)}$. Finally, it only remains to check column $\lambda_1-i+1$ of $A^{(i)}$. The point is that the only way to fall into cases
(2) or (3) is if one uses the tiebreak rule of Ryser's algorithm; this can only occur once. If no tiebreak occurs, case (1) occurs.
\end{proof}

\begin{definition}
\label{def:Alambdamureducible}
$A(\lambda,\mu)$ is \emph{reducible} if there exists some nontrivial subset $S$ of the columns such that the sum of the columns in $S$ forms a partition in ${\sf Par}_r$, and the sum of the columns not in $S$ forms a partition in ${\sf Par}_r$.
\end{definition}

\begin{proposition}
\label{prop:AredHB}
If $A(\lambda,\mu)$ is reducible, then $(\lambda,\mu)$ is reducible.
\end{proposition}
\begin{proof}
Let $S$ be as in Definition~\ref{def:Alambdamureducible}.  Let $\mu^\bullet$ be the vector obtained by adding all of the columns in $S$, and let $\mu^\circ = \mu - \mu^\bullet$ be the sum of all other columns. By the hypothesis that $A(\lambda,\mu)$ is reducible, $\mu^\bullet,\mu^\circ\in {\sf Par}_r$. Similarly, let $\lambda^\bullet$ and $\lambda^\circ$ be the partitions whose column multisets are $\{\lambda'_i:i\in S\}$, and $\{\lambda'_i:i\not\in S\}$ respectively.

Hence, 
$\mu = \mu^\bullet+\mu^\circ$ and $\lambda = \lambda^\bullet + \lambda^\circ$,
and these are both partitions.  Consider the submatrix of $A(\lambda,\mu)$ consisting of only the columns in $S$.  This is a $\{0,1\}$-matrix such that the sum of the entries in row $i$ is $\mu^\bullet_i$, and the sum of the entries in column $j$ is $(\lambda^\bullet)'_j$.  Therefore, by Proposition \ref{thm:0-1}, $(\lambda^\bullet,\mu^\bullet)\in \sf{Kostka}_r^{\mathbb Z}$.  By the same logic $(\lambda^\circ,\mu^\circ)\in \sf{Kostka}_r^{\mathbb Z}$, so $(\lambda,\mu)$ is reducible.
\end{proof}

\begin{example}
Continuing Example~\ref{exa:Jan8aaa}, $A(\lambda,\mu)$ is reducible by letting $S = \{2,3,4,8\}$.  This corresponds to 
\[\lambda^\bullet = (4,3,3,3,2,1), \mu^\bullet = (3,3,2,2,2,2,2), \lambda^\circ = (4,4,4,4,1,1), \text{\ and $\mu^\circ = (4,4,2,2,2,2,2)$.}\]
\end{example}

\begin{example}
\label{exa:conversenottrue}
The converse of Proposition~\ref{prop:AredHB} is not true.  Let $\lambda = (3,2,1)$, $\mu = (2,2,1,1)$. We have the decomposition 
\[\lambda^\circ = (1,1), \mu^\circ = (1,1), \lambda^\bullet = (2,1,1), \mu^\bullet = (1,1,1,1).\]  
However, one can verify that 
\begin{equation*}
A(\lambda,\mu) = \begin{bmatrix}
1 & 1 & 0\\
1 & 0 & 1\\
1 & 0 & 0\\
0 & 1 & 0\\
\end{bmatrix}
\end{equation*}
is not reducible.
\end{example}

\begin{lemma}
\label{lemma:underhood}
 Each column of $A^*(\lambda,\mu)$ is one of the following:
\begin{enumerate}
\item
all $0$'s, except for a single $1$;
\item
all $0$'s, except for a single $-1$ and a single $1$ below it; or
\item
all $0$'s, except for a single $-1$ and with a single $1$ above and a single $1$ below it.
\end{enumerate}
Moreover, the leftmost column is in the first case. The bottom row can not contain any $-1$'s.
\end{lemma}
\begin{proof}
This holds by the corresponding statements in Lemma~\ref{lemma:eachAcolumn}, together with (\ref{eqn:Jan12aaa}).
\end{proof}

Let 
\begin{equation}
\label{eqn:shapesequence}
\mu=\mu^{(0)}\supset\mu^{(1)}\supset \cdots \supset \mu^{(i)}\supset \cdots \supset\mu^{(\lambda_1)}=\emptyset
\end{equation}
denote the partitions of row sum vectors
for the $r\times (\lambda_1-i)$ leftmost submatrix of $A(\lambda,\mu)$. In our running example, the partitions are
\begin{multline}\nonumber
(7,7,4,4,4,4,4)\supset (7,6,4,4,4,4,4)\supset (6,5,4,4,4,3,3)\supset (5,4,4,3,3,3,3)\\ \nonumber
\supset (4,3,3,3,3,3,2)\supset (3,3,3,2,2,2,2)\supset (2,2,2,2,2,1,1)
\supset (1,1,1,1,1,1,0)\supset \emptyset\nonumber
\end{multline}

\begin{lemma}\label{colschange}
At step $i$, $\mu^{(i-1)}$ transforms to $\mu^{(i)}$ in precisely one of the following three ways: 
\begin{enumerate}
\item One column of $\mu^{(i-1)}$ is deleted (and the result is left-justified). 
\item The rightmost column of $\mu^{(i-1)}$ is shortened. 
\item One column of $\mu^{(i-1)}$ is shortened, and the column $C$ to its immediate right (equivalently, any column
of the same length as $C$) is deleted, and the result is
left-justified. The deleted column is strictly shorter than the shortened column's new length. 
\end{enumerate}
Moreover, if a column of $\mu^{(i-1)}$ is shortened, the new length of that column is strictly larger than the length of any column of $\mu^{(i)}$ to its right.
\end{lemma}

\begin{example}
In our running example, when $i=2$, the transition 
\[\mu^{(1)}=(7,6,4,4,4,4,4)\mapsto\mu^{(2)}=(6,5,4,4,4,3,3)\]
is obtained by applying (3): shortening column $4$ and deleting column $5$ (or column $6$).
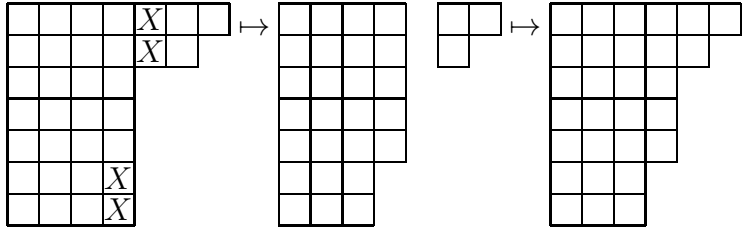
\[\tableau{\ & \ & \ & \ & X & \ & \ \\ \ & \ & \ & \ & X & \ \\ 
\ & \ & \ & \ \\
\ & \ & \ & \ \\
\ & \ & \ & \ \\
\ & \ & \ & X \\
\ & \ & \ & X } \mapsto 
\tableau{\ & \ & \ & \  &  & \  & \ \\ \ & \ & \ & \ & & \ \\ 
\ & \ & \ & \ \\
\ & \ & \ & \ \\
\ & \ & \ & \ \\
\ & \ & \ & \\
\ & \ & \ }
\mapsto
\tableau{\ & \ & \ & \  & \ & \ \\ \ & \ & \ & \ & \ \\ 
\ & \ & \ & \ \\
\ & \ & \ & \ \\
\ & \ & \ & \ \\
\ & \ & \ & \\
\ & \ & \ }
\]
\end{example}

\begin{proof}
Suppose we are executing the transition $A^{(i-1)}\mapsto A^{(i)}$ for $i\geq 1$.
We are moving (or using ``in place'') some $1$'s from the $r\times (\lambda_1-i+1)$ left submatrix of $A^{(i-1)}$ and having them
terminate in column $\lambda_1-i+1$.
If no tie-breaking occurs, then in fact the rows of the $1$'s used form an initial interval $[1,\ell]$ and there must exist a column of length $\ell$
in $\mu^{(i-1)}$ (specifically, the one that contains the southmost $1$ moved), and we fall into case (1). 

Now suppose tie-breaking occurs. Suppose that $\mu^{(i-1)}$ has largest part $m$, which occurs $t$ many times. If the number of
$1$'s moved is strictly fewer than $t$, we are in case (2). Finally, suppose the number of $1$'s moved is greater than $t$. Since tie-breaking is assumed, there is some part size $m'$, occurring $t'$ many times in $\mu^{(i-1)}$, for which Ryser's rule chooses
some southmost segment of $1$'s to move (as in the example $i=2$ above where $m'=4$ and $t'=5$). This means that Ryser's rule chooses $1$ from every row north of this block of $m'$. This puts us in case (3). 

The final sentence is clearly true if we are in case (2) since we are shortening the (weakly) shortest column. If instead we are
in case (3) and the shortening made the column equal to some other column of $\mu^{(i)}$ it is straightforward to see that would
imply no tie-breaking actually occurred, a contradiction of the analysis above.
\end{proof}

\begin{lemma}
\label{lemma:Jan18aaa}
The (exhaustive) cases \emph{(1)}, \emph{(2)}, \emph{(3)} of Lemma~\ref{colschange} in the transition $\mu^{(i-1)}\mapsto \mu^{(i)}$ respectively
correspond to the following $\pm 1$ values of $A^*(\lambda,\mu)$ in column $\lambda_1-i+1$:
\begin{itemize}
\item[(1)] $A^*(\lambda,\mu)_{\ell,\lambda_1-i+1}=1$ if the deleted column is of length $\ell$;
\item[(2)] $A^*(\lambda,\mu)_{\ell',\lambda_1-i+1}=-1, A^*(\lambda,\mu)_{\ell,\lambda_1-i+1}=1$ if the shortened column was of length $\ell$ and becomes of length $\ell'<\ell$;
\item[(3)] $A^*(\lambda,\mu)_{\ell'',\lambda_1-i+1}=1, A^*(\lambda,\mu)_{\ell',\lambda_1-i+1}=-1, A^*(\lambda,\mu)_{\ell,\lambda_1-i+1}=1$  if the shortened column was of length $\ell$ and becomes of length $\ell'<\ell$, and the deleted column is of length $\ell''(<\ell')$.
\end{itemize}
\end{lemma}
\begin{proof}
This follows from the definition of $A^*(\lambda,\mu)$ and by Ryser's algorithm.
\end{proof}

\begin{lemma}\label{alternation}
Every $-1$ in $A^*(\lambda,\mu)$ must have a $1$ to its left in the same row.  Furthermore, all entries between this $-1$ and 
the rightmost such $1$ are $0$'s. 
\end{lemma}
\begin{proof}
Suppose $A^*(\lambda,\mu)_{\ell',\lambda_1-i+1}=-1$. By Lemma~\ref{lemma:Jan18aaa} a column $C$ in $\mu^{(i-1)}$ of length $\ell>\ell'$ was shortened to length $\ell'$ at the step $A^{(i-1)}\mapsto A^{(i)}$ of Ryser's algorithm. The final sentence of Lemma~\ref{colschange} says that at each later stage no column of $\mu^{(i+k)}$ (for $k\geq 0$)
can shorten to length $\ell'$ until column $C$ is further shortened or deleted. Hence, by cases (2) and (3) of 
Lemma~\ref{lemma:Jan18aaa} no $-1$ will
appear left of our $-1$ until column $C$ is shortened/deleted. On the other hand, \emph{every} column is eventually shortened/deleted.
When this occurs, by Lemma~\ref{lemma:Jan18aaa}(1), the desired $1$ will appear. 
\end{proof}

\begin{definition}
$A^*(\lambda,\mu)$ is \emph{$*$-reducible} if there exists a nontrivial subset $S$ 
of the columns whose sum $v^*$ satisfies $0\leq v^*_i \leq \mu^*_i$ for all $i$.
\end{definition}

\begin{proposition}
\label{prop:A*iffA}
$A^*(\lambda,\mu)$ is $*$-reducible if and only if $A(\lambda,\mu)$ is reducible.
\end{proposition}
\begin{proof}
Suppose $A^*(\lambda,\mu)$ is $*$-reducible; let $S$ be the witnessing subset. Let $v$ be the sum of the columns
of $S$ in $A(\lambda,\mu)$. Now 
\[v^*\in {\mathbb Z}_{\geq 0}^r \iff v\in {\sf Par}_r.\] 
Similarly let $v^c$ be the sum of the
columns of $S^c$ in $A(\lambda,\mu)$. We have 
\[\mu^*-v^*\in {\mathbb Z}_{\geq 0}^r \iff v^c\in {\sf Par}_r.\] 
Thus,
$S$ witnesses $A(\lambda,\mu)$ being reducible. The converse proof reverses this argument.
\end{proof}

\begin{proposition}
\label{claim:sources}
$G(\lambda,\mu)$ has exactly $\mu^*_i$ sources in row $i$, and they are all $1$'s in $A^*(\lambda,\mu)$.
\end{proposition}
\begin{proof}
By Lemma~\ref{alternation} and the definition of $G(\lambda,\mu)$ any source must correspond to 
a $1$. Now since the sum of the entries in row $i$ of $A^*(\lambda,\mu)$ is $\mu_i^*$. Each $-1$ in that row 
can be paired with the $1$ closest to its left (which exists by Lemma~\ref{alternation}). There is a surplus of
$\mu_i^*$-many unpaired $1$'s. These are the $\mu_i^*$-many sources.
\end{proof}

\begin{proposition}\label{partA}
The undirected graph structure of
$G(\lambda,\mu)$ is a forest that is planar (with respect to the specified embedding).
\end{proposition}
\begin{proof}
For the forest claim, we show there is no undirected cycle. Assume otherwise and such a 
cycle ${\mathcal C}$ exists in $G(\lambda,\mu)$. No horizontal arc of $G(\lambda,\mu)$ points to the right. Therefore, it follows
that some vertex $v$ of ${\mathcal C}$ has two distinct arcs pointing away from $v$. But, by construction, no vertex of $G(\lambda,\mu)$
has this property, a contradiction.

Assume for the sake of contradiction that two edges of $G(\lambda,\mu)$ cross over each other.  Locally, the two
possible configurations look like:
\begin{equation*}
A^*(\lambda,\mu) = \begin{bmatrix}
\rn{11'}{\cdots} & \rn{12'}{\cdots} & \rn{13'}{\cdots} & \rn{14'}{\cdots} & \rn{15'}{\cdots} \\
\rn{21'}{\cdots} & \rn{22'}{\cdots} & \rn{23'}{1} & \rn{24'}{\cdots} & \rn{25'}{\cdots}\\
\rn{31'}{\cdots} & \rn{32'}{\cdots} & \rn{33'}{\cdots} & \rn{34'}{\cdots} & \rn{35'}{\cdots} \\
\rn{41'}{\cdots} & \rn{42'}{1} & \rn{43'}{\cdots} & \rn{44'}{-1} & \rn{45'}{\cdots}\\
\rn{51'}{\cdots} & \rn{52'}{\cdots} & \rn{53'}{\cdots} & \rn{54'}{\cdots} & \rn{55'}{\cdots} \\
\rn{61'}{\cdots} & \rn{62'}{\cdots} & \rn{63'}{-1} & \rn{64'}{\cdots} & \rn{65'}{\cdots}\\
\end{bmatrix},
\begin{bmatrix}
\rn{11}{\cdots} & \rn{12}{\cdots} & \rn{13}{\cdots} & \rn{14}{\cdots} & \rn{15}{\cdots} \\
\rn{21}{\cdots} & \rn{22}{\cdots} & \rn{23}{-1} & \rn{24}{\cdots} & \rn{25}{\cdots}\\
\rn{31}{\cdots} & \rn{32}{\cdots} & \rn{33}{\cdots} & \rn{34}{\cdots} & \rn{35}{\cdots} \\
\rn{41}{\cdots} & \rn{42}{1} & \rn{43}{\cdots} & \rn{44}{-1} & \rn{45}{\cdots}\\
\rn{51}{\cdots} & \rn{52}{\cdots} & \rn{53}{\cdots} & \rn{54}{\cdots} & \rn{55}{\cdots} \\
\rn{61}{\cdots} & \rn{62}{\cdots} & \rn{63}{1} & \rn{64}{\cdots} & \rn{65}{\cdots}\\
\end{bmatrix}
\end{equation*}
\begin{tikzpicture}[overlay,remember picture]
\draw [->, color=blue, line width=0.5mm] (63) -- (23);
\draw [->, color=blue, line width=0.5mm] (44) -- (42);
\draw [->, color=blue, line width=0.5mm] (23') -- (63');
\draw [->, color=blue, line width=0.5mm] (44') -- (42');
\put(170,0){$\lambda_1-j+1$}
\put(312,0){$\lambda_1-j+1$}
\put(410,40){$\ell'$}
\put(410,11){$b$}
\put(410,69){$a$}
\end{tikzpicture}

In either case, the horizontal edge is in row $\ell'$ and the vertical edge is in column $\lambda_1-j+1$.  
The horizontal edge implies, by Lemma~\ref{lemma:Jan18aaa} that
$\mu^{(j-1)}$ has a column of length $\ell'$.  However, the vertical edge implies, respectively that 
in the transition from $\mu^{(j-1)}$ to $\mu^{(j)}$,
\begin{enumerate}
\item
a column of length $b'>b$ is shortened to length $b>\ell'$, and a column of length $a<\ell'$ is deleted; or
\item
a column of length $b>\ell'$ is shortened to length $a<\ell'$.
\end{enumerate}
By Lemma~\ref{colschange}, neither of these is possible, a contradiction. Hence $G(\lambda,\mu)$ is planar.
\end{proof}

\begin{lemma}
\label{lemma:lefttoright}
Let $u,v$ be vertices of $G(\lambda,\mu)$ with ${\sf col}(v)<{\sf col}(u)$. Let $u=v_0,v_1,\ldots,v_N=v$ be distinct vertices
of an undirected path $P$ from $u$ to $v$. No step of this path moves right (\emph{i.e.,} ``backwards'') on a horizontal arc $v_{i}\leftarrow v_{i+1}$ where ${\sf col}(v_{i+1})>{\sf col}(v)$.
\end{lemma}
\begin{proof}
Suppose a path $P$ exists with this backwards step $v_i\leftarrow v_{i+1}$. Then since ${\sf col}(v_{i+1})>{\sf col}(v)$, there must exist
a later step that uses a leftward (\emph{i.e.}, ``forward'') move along a horizontal arc $v_{j+1}\leftarrow v_j$ (for some $j>i$). Since this
orientation change occurs, for some $i< k\leq j$, $P$ uses a backwards step $v_{k-1}\leftarrow v_{k}$ followed by a forward step $v_{k}\rightarrow v_{k+1}$ (possibly on a vertical arc). Since all the vertices $v_t$ are distinct, we conclude
$v_k$ has out-degree at least two. However no vertex in $G(\lambda,\mu)$ has this property, a contradiction.
\end{proof}

\begin{lemma}
\label{lemma:connected}
$G(\lambda,\mu)$ is connected if and only if there is a horizontal arc emanating from every column in $G(\lambda,\mu)$ except the leftmost one.
\end{lemma}
\begin{proof} ($\Leftarrow$) Suppose there is a horizontal arc beginning in every column $C$ in $G(\lambda,\mu)$ that is not the leftmost one.  This means, for every vertex $v\in C$ there is a directed path from $v$ to some vertex $u$ such that ${\sf col}(u)<{\sf col}(v)$.  By iterating this, we obtain a directed path from $v$ to $s$, the unique vertex in column $1$.  Our choice of $v$ was arbitrary, so there is a directed path from any vertex to $s$, and so $s$ is in the same connected component as all other vertices. Hence $G(\lambda,\mu)$ is connected.

($\Rightarrow$) Suppose there is some (not leftmost) column $C$ such that $C$ has no horizontal edge emanating from it.
We wish to show $G(\lambda,\mu)$ is not connected. Assume otherwise. Let $u\in C$, and say that $v$ is a vertex with ${\sf col}(v)<{\sf col}(u)$.  Since $G(\lambda,\mu)$ is connected, there exists an undirected path $P$ from $u$ to $v$.
  Since column $C$ has no horizontal edges coming out of it, and all horizontal edges point left, the first horizontal step in $P$ is to the right.  Since ${\sf col}(v)<{\sf col}(u)$, Lemma~\ref{lemma:lefttoright} says $P$ does not exist, a contradiction.
    \end{proof}

\begin{lemma}
\label{disconnectedcase}
If $G(\lambda,\mu)$ is not connected, any connected component $G'$ is a conservative subtree.
\end{lemma}
\begin{proof}
In view of (C.1), it suffices to prove that whenever a vertical arc is in $G'$, all vertical arcs of $G(\lambda,\mu)$, in the same
column, also appear. This follows from Lemma~\ref{lemma:underhood}, the connectedness of $G'$ and $G(\lambda,\mu)$'s definition.
\end{proof}

\begin{proposition}
\label{lemma:hascon}
If $G(\lambda,\mu)$ has a conservative subtree then $A^*(\lambda,\mu)$ is $*$-reducible.
\end{proposition}
\begin{proof} Let $G'$ be a conservative subtree of $G(\lambda,\mu)$, and let $S$ be the set of columns containing vertices of $G'$.  Suppose $u$ is a vertex that is not a sink or source. Since $G'$ is by definition
connected, by construction of $G(\lambda,\mu)$, $u$ is incident to one horizontal arc. Let $u'$ be the other vertex of this arc.
In any directed path of $G'$, the vertices alternate from $-1$ to $1$. Therefore, the entries of 
$A^*(\lambda,\mu)$ corresponding to $u$ and $u'$ are of opposite sign. Hence, if $v^*$ is the sum of the columns of $S$, the entries of
$A^*(\lambda,\mu)$ associated to $u$ will cancel. 

It remains to  check the contribution of sources and sinks to $v^*$. There are two cases.

{\sf Case 1:} ($G'$ is of type (C.1)) Every vertex of $G'$ that is not a source of $G(\lambda,\mu)$ is
incident to exactly one horizontal arc, so when adding up the columns of $S$, the only $\pm 1$ values that do not cancel out are the sources of $G'$.  However, these are all sources of $G(\lambda,\mu)$, so by Proposition~\ref{claim:sources} they contribute at most $\mu^*_i$ $1$'s to row $i$.  Hence 
$v^*_i\leq \mu^*_i$. 
This says $A^*(\lambda,\mu)$ is $*$-reducible provided $S\subsetneq [\lambda_1]:=\{1,2,\ldots,\lambda_1\}$. Indeed this is the case since if $S=[\lambda_1]$
then $G'$ contains a vertex in each column. Then since all vertices in a column are connected, we would see that
$G'=G(\lambda,\mu)$ which is a contradiction.

{\sf Case 2:} ($G'$ is of type (C.2))
The only vertices in $G'$ that are not incident to a horizontal arc in $G'$ are the sink and the sources.  
All other vertices of $G'$ are incident to a unique horizontal arc.
By the definition of (C.2), one of the sources is a $1$ that is in the same row as the unique sink, which is a $-1$ in $A^*(\lambda,\mu)$.  As a result, the only values that do not cancel out when adding the columns in $S$ are the other sources of $G'$.  Since these vertices are also sources of $G$, by appeal
to Proposition~\ref{claim:sources} (as in {\sf Case 1}), we see  $A^*(\lambda,\mu)$ is $*$-reducible, provided $S\subsetneq [\lambda_1]$.
To see this proper containment holds, notice that the first column of $A^*(\lambda,\mu)$ contains a single $1$ and hence must correspond to a sink vertex $q$. Since $G'$ has a unique sink (which is a $-1$), $q$ is not in $G'$ and thus $1\not\in S$.
\end{proof}

In fact the converse is true; see Theorem~\ref{thm:someifftrue}.

\begin{example}
Our conservative subtree uses $S=\{2,3,4,8\}$. The reader can check that with this $S$, $A^*(\lambda,\mu)$ is $*$-reducible,
with 
\[v=(0,1,0,0,0,0,2)^T \leq \mu^*=(0,3,0,0,0,0,4)^T.\] 
\end{example}

\noindent\emph{Conclusion of the proof of Theorem~\ref{thm:bigone}:} The first claim about the planar forest structure of $G(\lambda,\mu)$ is Proposition~\ref{partA}. The second claim follows from Proposition~\ref{lemma:hascon}, Proposition~\ref{prop:A*iffA}, and
Proposition~\ref{prop:AredHB}.  \qed

\section{Proof of the Width Bound (Theorem~\ref{thm:bound})} \label{sec:2.2} Our main application of Theorem~\ref{thm:bigone} is to prove
Theorem~\ref{thm:bound}. First let us state some facts we will need:

\begin{lemma}
\label{claim:1VertEdge}
If a column of $G(\lambda,\mu)$ of (matrix index) $c$ contains exactly one vertical arc, and in addition the arc's northmost endpoint is in row $n$, 
there is no horizontal arc $y\leftarrow x$ (necessarily leftward) such that ${\sf row}(x)={\sf row}(y)<n$ with
${\sf col}(x)> c$ and ${\sf col}(y)\leq c$.
\end{lemma}
\begin{proof}
Since there is only a single vertical arc, this means that at step $\lambda_1-c+1$ of Ryser's
algorithm, a column $C$ in $\mu^{(\lambda_1-c)}$ is shortened to length $n$ (in $\mu^{(\lambda_1-c+1)}$), and no other column was deleted. Therefore, we must fall into Lemma~\ref{colschange}(2). Thus
$C$ is the rightmost and shortest column in $\mu^{(\lambda_1-c)}$. $C$ remains the shortest in
$\mu^{(\lambda_1-c+1)}$. However, existence of the arc $y\leftarrow x$
implies that at step $\lambda_1-c+1$ of the algorithm, there is a shorter column in $\mu^{(\lambda_1-c)}$  of length ${\sf row}(x)={\sf row}(y)<n$. This contradicts the previous sentence.
\end{proof}

\begin{lemma}
\label{claim:SourceBlock}
Suppose $G(\lambda,\mu)$ has a source $s$. There is no vertical arc  of the form:
\begin{itemize} 
\item[(a)] $\underset{x}{\overset{y}{\uparrow}}$ with ${\sf row}(x)> {\sf row}(s)$ and ${\sf row}(y)\leq {\sf row}(s)$; or
\item[(b)] $\underset{y}{\overset{x}{\downarrow}}$ with ${\sf row}(x)<{\sf row}(s)$ and ${\sf row}(y)\geq {\sf row}(s)$,
\end{itemize}
such that ${\sf col}(x)={\sf col}(y)>{\sf col}(s)$.
\end{lemma}
\begin{proof}
By Proposition~\ref{claim:sources}, since $s$ is a source, it corresponds to a $1$.  
We claim that at step $t_1=\lambda_1-{\sf col}(x)+1$ of Ryser's algorithm, there is a
column $C$ of length ${\sf row}(s)$ that is unaltered by the step. Indeed, what we know is that at step $t_2=\lambda_1-{\sf col}(s)+1$
there is a column of length ${\sf row}(s)$. If no such column exists at step $t_1$, then between step $t_1$ and step $t_2$ some longer
column must have been shortened, at step $t'$, to length ${\sf row}(s)$ which is then shortened or deleted at step $t_2$. 
This implies $A^*(\lambda,\mu)_{{\sf row}(s),\lambda_1-t'+1}=-1$ and 
$A^*(\lambda,\mu)_{{\sf row}(s),\lambda_1-t+1}=0$ for $t'<t<t_2$. However, we know
$A^*(\lambda,\mu)_{{\sf row}(s),\lambda_1-t_2+1}=1$ which implies $s$ is not a source, a contradiction.

Suppose a vertical arc of type (a) occurs. Then examining the possibilities of Lemma~\ref{lemma:Jan18aaa}, we conclude
that some column of length ${\sf row}(x)>{\sf row}(s)$ in $\mu^{(t_1-1)}$ was shortened and is now (in $\mu^{(t_1)}$) 
shorter/equal in length than $C$. This contradicts the last sentence of Lemma~\ref{colschange}. Hence
an arc of type (a) cannot occur.

Now suppose a vertical arc of type (b) occurs. By Lemma~\ref{lemma:Jan18aaa} there must be another vertical arc $\underset{x'}{\overset{y}{\uparrow}}$ in ${\sf col}(x)={\sf col}(y)$. Thus at step $t_1$ we must be in case (3) of Lemma~\ref{colschange}.
However, in $\mu^{(t_1)}$, the shortened column $S$ has length ${\sf row}(y)>{\sf row}(s)$ and the column $D$ deleted
in $\mu^{(t_1-1)}$ is of length
${\sf row}(x)<{\sf row}(s)$. Therefore $D$ cannot be of length equal to the one to the immediate right of $S$ (since $C$ is in-between), as demanded by Lemma~\ref{colschange}(3), a contradiction. Hence no arc of type (b) can occur either.
\end{proof}

First suppose that $\lambda_1>r$. 
If $G(\lambda,\mu)$ is not connected, let $G'$ be any connected component. Then $G'$ is a 
conservative subtree by Lemma~\ref{disconnectedcase} and we are done by Theorem~\ref{thm:bigone}.

Therefore we may assume $G(\lambda,\mu)$ is connected. Therefore, Lemma~\ref{lemma:connected} says every column other than the leftmost one has a leftward horizontal arc coming out of it.  That is, there at least $\lambda_1-1$ horizontal arcs.
No horizontal arc appears in the bottom row (since $A^*(\lambda,\mu)$ has no $-1$'s in that row). Hence the arcs appear in
at most $r-1$ rows. Since we assume $\lambda_1-1>r-1$, by pigeonhole, some row contains two horizontal arcs $u\to u'$ and $v\to v'$.
We may suppose $v\to v'$ is to the left of $u\to u'$.  
 
 \begin{lemma}
\label{lemma:samecol}
If there exist vertices $u',v\in G(\lambda,\mu)$ such that $v$ is a $-1$, $u'$ is a $1$, $\sf{row}(v)=\sf{row}(u')$ and ${\sf col}(v)< {\sf col}(u')$, then $G(\lambda,\mu)$ has a conservative subtree.
\end{lemma}
\begin{proof}
If $G(\lambda,\mu)$ is not connected, then we are done by Lemma~\ref{disconnectedcase}.  Therefore, we can assume that $G(\lambda,\mu)$ is connected.  

\begin{claim}
\label{claim:Jan20fff}
There is a directed path from $u'$ to $v$ in $G(\lambda,\mu)$.
\end{claim}
\noindent
\emph{Proof of Claim~\ref{claim:Jan20fff}:} Suppose not. Since $G(\lambda,\mu)$ is connected, 
there is an \emph{undirected} path (with distinct vertices) $P^+$ from $u'$ to $v$. We assert that $P^+$ looks like the one in Figure~\ref{Jan20ggg} (or its vertical reflection, for which our argument is the same, only flipped, except that we
apply the definition of $G(\lambda,\mu)$ whenever we refer to Lemma~\ref{claim:1VertEdge} in that case).  More precisely, let $P$ be the initial part of $P^+$ from $u'$ to $1$ (in Figure~\ref{Jan20ggg}, this is represented by the dashed blue lines).  We claim
$P$ is a \emph{directed path} from $u'$ to $1$. These assertions about $P^+$ and $P$ are straightforward by inspecting the
definition of $G(\lambda,\mu)$ and Lemmas~\ref{lemma:underhood}, \ref{alternation} and~\ref{lemma:lefttoright}.
Let $R$ be the region (strictly) bounded from above by $P$, (weakly) bounded from the left and right by the columns of $v$ and $u'$, respectively, and unbounded from below.

\begin{figure}[t]
\begin{tikzpicture}[scale=0.75]
\fill [color=lightgray] (-1,-3) -- (-1,3) -- (0,3) -- (0,4.5) -- (2,4.5) -- (2,-2.5) -- (4,-2.5) -- (4,-3) -- (0,-3);
\draw [color=white, fill=white] (2,4.5) circle (1.5ex);
\draw [color=white, fill=white] (-1,-1) circle (1.5ex);
\node (v) at (-1,-1) {$v$};
\node (n) at (-3,1.5) {$-1$};
\node (o) at (-3,3) {$1$};
\node (o2) at (-3,0) {$1$};
\node (P) at (2,4.5) {$P$};
\node (up) at (4,-1) {$u'$};
\node (s) at (1,2.5) {$s$}; 
\node at (0.5,-1) {$R$};
\draw [->, color=black, line width=0.5mm, dashed] (v) --(-2,-1) -- (-2,0) -- (o2);
\draw [->, color=black, line width=0.5mm] (o2) -- (n);
\draw [->, color=blue, line width=0.5mm] (o) -- (n);
\draw [->, color=blue, line width=0.5mm,dashed] (up) -- (4,-2.5) -- (2,-2.5) -- (P) -- (0,4.5) -- (0,3) --(o);
\draw [->, color=blue, line width=0.5mm] (s) -- (1,1.5);
\draw [->, color=red, line width=0.5mm,dashed] (1,1.5) -- (-1,1.5) -- (v);
\end{tikzpicture}
\caption{Proof of Claim~\ref{claim:Jan20fff} \label{Jan20ggg}. The column containing the depicted $\pm 1$'s is strictly left
of ${\sf col}(v)$.}
\end{figure} 

Consider the induced subgraph ${\overline G}$ of $G(\lambda,\mu)$ consisting of all vertices that have a directed path to $v$.  Let $s$ be the unique vertex of ${\overline G}$ in $R$ that is northmost then rightmost among all choices (such an $s$ exists since $v$ itself
is in $R$).

Observe that $s$ corresponds to a $1$ in $A^*(\lambda,\mu)$: otherwise if it corresponds to a $-1$, it must have a vertical arc $e$
coming into it (Lemma~\ref{lemma:underhood}). If $e$ comes from above, and is entirely inside $R$, then we 
contradict the maximality of $s$. If the other endpoint $x$ is on $P$ then there is a directed path 
$u'\dashrightarrow x\dashrightarrow s \dashrightarrow v$,  
which contradicts the assumption of this argument. Finally if neither is true, we contradict 
Theorem~\ref{thm:bigone} (planarity of
$G(\lambda,\mu)$). If $e$ only comes into $s$ from below, then by Lemma~\ref{lemma:underhood} this is the unique vertical arc
in its column, and we contradict Lemma~\ref{claim:1VertEdge}.

By a similar argument to the previous paragraph, $s$ is a source, since otherwise there is a horizontal arc coming into $s$ from
the right, and we would either violate the maximality of $s$, planarity of $G(\lambda,\mu)$, or the initial assumption that there is no path from $u'$ to $v$.

We assert ${\sf row}(s)<{\sf row}(v)$: Since $v$ corresponds to a $-1$, by 
Proposition~\ref{claim:sources},
it is not a source, and hence has a vertical arc $e$ entering it (it cannot be a horizontal arc because of the definition of $G(\lambda,\mu)$). By Lemma~\ref{claim:1VertEdge} and the presence of $P$, in fact $v$ must have two vertical arcs entering $v$. Thus one of
these points down, and the assertion follows.

Finally, since $P$ starts at 
\[{\sf row}(u')={\sf row}(v)>{\sf row}(s)\]
 and rises strictly above ${\sf row}(s)$ (since $s$ is in $R$),
 there is a vertical edge of $P$ which together with the source $s$ violates Lemma~\ref{claim:SourceBlock}, completing the proof
 of the Claim.\qed

 Let $G'$ be the subgraph of $G(\lambda,\mu)$ induced by $u'$ and all vertices that have a directed path to $v$ that does not go through $u'$. We claim that this
 $G'$ falls into (C.2) of the definition of conservative subtree. $G'$ is a nontrivial subgraph of $G(\lambda,\mu)$ since $v$ must point to a vertex that is not included in $G'$, because we assume $v$ is a $-1$ and by Lemma~\ref{alternation}. 
  By Claim~\ref{claim:Jan20fff}, $G'$ is connected (as an undirected graph). Notice that if any vertical arc appears in $G'$, it follows
 from Lemma~\ref{lemma:underhood} that all vertical arcs of $G(\lambda,\mu)$ in that column do. By construction $v$ is the unique sink of $G'$  and $u'$ provides the source in the same row, as demanded by (C.2). Finally, if $s'$ is a source of $G'$, it must be a source
 of $G(\lambda,\mu)$ since otherwise any vertex $x$ such that $x\to s'$ is an arc would also be in $G'$.  Thus $G(\lambda,\mu)$ has a conservative subtree, and we are done with the proof of the Lemma.
\end{proof}

By applying Lemma~\ref{lemma:samecol} to $u'$ and $v$, we have that $G(\lambda,\mu)$ has a conservative subtree, and so we are done by Theorem~\ref{thm:bigone}. We are thus finished the proof in the case $\lambda_1>r$.

Now suppose $\lambda_1=r$ and $(\lambda,\mu)\in {\sf Kostka}_r^{\mathbb Z}$ is in the Hilbert basis. We claim that 
$\mu$ is a rectangle. Suppose not. We may assume $G(\lambda,\mu)$ is connected by Lemma~\ref{disconnectedcase} and Theorem~\ref{thm:bigone}. If there exist two horizontal arcs in the same row, we are done by Lemma~\ref{lemma:samecol}.
Otherwise, there is \emph{exactly} one horizontal arc in each row (except the last one):
to see this use Lemma~\ref{lemma:connected} and the assumption that there are $\lambda_1-1=r-1$ columns. But since $\mu$ is not a rectangle, by Proposition~\ref{claim:sources}, one of these rows also contains a source of $G(\lambda,\mu)$. Let $u'$ be the source
and let $v'\leftarrow v$ be the horizontal arc in the same row. Hence $v$ is a $-1$ so there must be a vertical arc in ${\sf col}(v)$
that ends at $v$. Thus, if ${\sf col}(v)>{\sf col}(u')$ then we contradict Lemma~\ref{claim:SourceBlock}. Therefore ${\sf col}(v)<{\sf col}(u')$. Now apply Lemma~\ref{lemma:samecol} to obtain a conservative subtree, which implies $(\lambda,\mu)$ is not a Hilbert basis
element (Theorem~\ref{thm:bigone}), a contradiction.

Thus, the next lemma finishes the $\lambda_1=r$ case.

\begin{lemma}
Suppose $(\lambda,\mu)\in {\sf Kostka}_r^{\mathbb Z}$, $\lambda_1 = r$, and $\mu$ is a rectangle. If $\lambda$ is not a rectangle, then $(\lambda,\mu)$ is not a Hilbert basis element. 
\end{lemma}

\begin{proof}
Let $s$ be the length of the columns of $\mu$. Write the lengths of the columns of $\lambda$ as 
\[a_1\ge a_2 \ge \hdots \geq a_r.\] 
Since $\mu'\geq_{\sf Dom}\lambda'$, 
$a_1\le s$. 
Moreover, if $a_1 = s$ then $(\lambda,\mu)$ has a decomposition with $(\lambda^\bullet,\mu^\bullet) = ((1)^s,(1)^s)$. So let us assume that $a_1<s$. 

There exists some $j$ such that $a_1\ne a_j$. Reorder the column lengths of $\lambda$ as 
\[b_1 = a_1, b_2 = a_j, b_3 = a_2, \hdots, b_{j} = a_{j-1} ,b_{j+1} = a_{j+1}, \hdots, b_r = a_{r}.\] 
Consider the sequence of $r+1$ integers 
\begin{equation}\label{eqn:seq}
b_1,b_2,b_1+b_2, b_1+b_2 +b_3, \hdots, \sum_{k=1}^{r-1} b_k, \sum_{k=1}^r b_k = \mu_1\cdot s. 
\end{equation}
By pigeonhole, since $r+1> s$ there must be two of these integers with the same remainder modulo $s$. Moreover, 
\[b_1\not\equiv b_2 \text{\ mod $s$}\] 
since $0<b_2<b_1<s$. So the difference of our two congruent integers from (\ref{eqn:seq}) is of the form $\sum_{k\in I} b_k$. Then $(\lambda,\mu)$ has a decomposition where $\lambda^\bullet$ is given columns of lengths $b_k, k\in I$, and $\mu^\bullet$ has  $\frac{1}{s}\sum_{k\in I} b_k\in {\mathbb Z}$ columns of length $s$ each. Since $a_1=b_1\leq s$ it follows that
\[(\mu^\bullet)' \ge_{\sf Dom} (\lambda^\bullet)' \text{\  and $(\mu^\circ)' \ge_{\sf Dom} (\lambda^\circ)'$,}\] 
where $(\lambda^\circ,\mu^\circ) = (\lambda,\mu) - (\lambda^\bullet, \mu^\bullet)$. 
\end{proof}

The proof of Theorem~\ref{thm:bound} is now complete.\qed

Using the results of this section and the last we derive an additional result:

\begin{theorem}\label{thm:someifftrue}
Ryser's canonical matrix $A(\lambda,\mu)$ is reducible if and only if $G(\lambda,\mu)$ has a conservative subtree.
\end{theorem}
\begin{proof}
First, let us establish the converse to Proposition~\ref{lemma:hascon}. That is we claim that if $A^*(\lambda,\mu)$ is $*$-reducible then
$G(\lambda,\mu)$ has a conservative subtree. If $A^*(\lambda,\mu)$ is $*$-reducible, 
there is some subset $\emptyset\subset S\subset [\lambda_1]$ of the columns of $A^*(\lambda,\mu)$ whose sum $v^*$ satisfies $0\leq v^*_i\leq \mu^*_i$ for all $i$.  We may assume that $1\not\in S$, as otherwise we can consider $[\lambda_1]\setminus S$.

If $G(\lambda,\mu)$ is not connected, then we are done, as any connected component of $G(\lambda,\mu)$ is a conservative subtree, by 
Lemma~\ref{disconnectedcase}.  Therefore, we can assume that $G(\lambda,\mu)$ is connected.  Lemma~\ref{lemma:connected} says that there must be a horizontal edge starting in every column besides column $1$, and since every horizontal edge starts at a $-1$, every column except column $1$ must have a $-1$ in it.  Let $c=\min(S)$.  Since $c>1$, there is some vertex in column $c$, denoted $w$, that is a $-1$.  Since $v^*_{{\sf row}(w)}\geq 0$, there must be some vertex $u'$ that is a $1$ such that ${\sf row}(w) = {\sf row}(u')$ and ${\sf col}(u')\in S$, so in particular ${\sf col}(u') > {\sf col}(w)$.  As a result, Lemma~\ref{lemma:samecol}  implies that there is a conservative subtree, completing the argument.

The result follows by combining the above argument with Proposition~\ref{lemma:hascon} and
Proposition~\ref{prop:A*iffA}.
\end{proof}

Theorem~\ref{thm:someifftrue} combined with Example~\ref{exa:conversenottrue} shows that $(\lambda,\mu)$ being reducible does not imply $G(\lambda,\mu)$
has a conservative subtree. 

\section{Classification of the rays of the Kostka cone}\label{sec:4}

We break up the proof into two propositions. The first verifies that our candidates are indeed extremal rays. The second shows that all extremal rays are of this form. 

\begin{proposition}
\label{prop:firstextremal}
Let $a,b,\ell$ satisfy $a\ge b>0$, $\ell\ge 0$, and $\ell+a\le r$. Then 
$$
(\lambda,\mu) = \left(\underbrace{a,\hdots,a}_{b+\ell},0,\hdots,0; \underbrace{a,\hdots,a}_{\ell},\underbrace{b,\hdots,b}_{a},0\hdots,0\right)
$$
gives an extremal ray of ${\sf Kostka}_r$.
\end{proposition}

\begin{proof}
The only way to write $(\lambda,\mu)$ as a sum of two real partitions is 
$$
\sum_{i=1}^2 \left(\underbrace{a_i,\hdots,a_i}_{b+\ell},0,\hdots,0; \underbrace{c_i,\hdots,c_i}_{\ell},\underbrace{b_i,\hdots,b_i}_{a},0\hdots,0\right)
$$
where $a = a_1+a_2 = c_1+c_2$ and $b = b_1+b_2$ and $c_1\ge b_1,c_2\ge b_2$. Suppose that both summands belong to the cone. Then \[\ell\cdot a_1 \ge \ell\cdot c_1\] 
by (\ref{eqn:dominance}) for $t=\ell$; likewise 
\[\ell\cdot a_2\ge \ell\cdot c_2.\] 
These two inequalities add to make $\ell a = \ell a$, so they both hold with equality. Thus $a_1 = c_1$ and $a_2 = c_2$ (or $\ell=0$, in which case the $c_i$ don't exist). Now we also know
\begin{align*}
a_1(\ell+b) &= \ell\cdot c_1+a\cdot b_1,\\
a_2(\ell+b) &= \ell\cdot c_2+a\cdot b_2,
\end{align*}
which rearrange to $a_1/b_1 = a/b$ and $a_2/b_2 = a/b$ (independent of $\ell$). Thus our two summands are parallel. 
\end{proof}

\begin{lemma}\label{lem:easycase}
Suppose $(\lambda,\mu)$ lies on an extremal ray of ${\sf Kostka}_r$ and that all inequalities (\ref{eqn:dominance}) hold strictly for $1\leq t<r$. Then $(\lambda,\mu)$ is parallel to $(a^b,b^a)$ for suitable integers $r\ge a\ge b>0$.
\end{lemma}

\begin{proof}
Define $b$ to be the length of $\lambda$ and $a$ that of $\mu$. Consider the families of real partitions 
$$
\lambda_{\pm\epsilon}:=\lambda\pm\frac{1}{b}\epsilon(1^b)
$$
and 
$$
\mu_{\pm\epsilon}:=\mu\pm\frac{1}{a}\epsilon(1^a).
$$

Then $|\lambda_{+\epsilon}| = |\mu_{+\epsilon}|$ (same for $-\epsilon$), and $\lambda_{\pm\epsilon}, \mu_{\pm\epsilon}$ are all real partitions for $\epsilon>0$ small enough. Furthermore, for $\epsilon$ sufficiently small, all inequalities (\ref{eqn:dominance}) hold for the pairs $(\lambda_{\pm\epsilon},\mu_{\pm\epsilon})$. Since $(\lambda,\mu)$ is extremal, it must be true that $\lambda_{+\epsilon}$ and $\lambda_{-\epsilon}$ are parallel and, furthermore, parallel to $\lambda$. This implies that, for any $j\le b$, 
\begin{align*}
(\lambda_j+\epsilon/b)/\lambda_j &= (\lambda_1+\epsilon/b)/\lambda_1,
\end{align*}
hence $\lambda_1=\lambda_j$. Similarly, $\mu_1 = \mu_j$ for any $j\le a$. Now scale $(\lambda,\mu)$ so that $\lambda_1 = a$. Since 
\[a \mu_1 = |\mu| = |\lambda| = ba,\] 
we must have $b = \mu_1$, and the statement is proven.
\end{proof}

 \begin{definition}
Suppose $\nu=(\nu_1\ge \hdots \ge \nu_s)$ and $\xi = (\xi_1\ge\hdots \ge \xi_t)$ are two partitions. If $\nu_s\ge \xi_1$, then 
$$
(\nu_1\ge\hdots\ge\nu_s \ge \xi_1\ge \hdots \ge \xi_t)
$$
is a partition of $s+t$ parts, the \emph{concatenation} of $\nu$ and $\xi$, which we denote by $\nu \concat \xi$. 
\end{definition}

Combined with Proposition~\ref{prop:firstextremal} and Lemma~\ref{lem:easycase}, the next statement completes the proof of
 Theorem~\ref{thm:rays}. 
 
\begin{proposition}
\label{prop:thefinalprop}
Suppose $(\lambda,\mu)$ is an extremal ray of ${\sf Kostka}_r$ where at least one inequality (\ref{eqn:dominance}) with $t<r$ holds with equality. Then $(\lambda,\mu)$ is parallel to 
$$
\left(\underbrace{a,\hdots,a}_{b+\ell},0,\hdots,0; \underbrace{a,\hdots,a}_{\ell},\underbrace{b,\hdots,b}_{a},0\hdots,0\right)
$$
for suitable $a\ge b>0$ and $\ell\ge 0$. 
\end{proposition}

\begin{proof}
We proceed by induction on $r\geq 1$. The base case $r=1$ is trivial. Now suppose $r>1$. 
Take $(\lambda,\mu)$ to lie on an extremal ray of ${\sf Kostka}_r$. For simplicity, assume that $(\lambda,\mu)\in {\sf Kostka}_r^\mathbb{Z}$ (since the inequalities defining ${\sf Kostka}_r\subseteq \mathbb{R}^{2r}$ have rational coefficients, this is possible; \emph{cf}. \cite[\S 16.2]{Schrijver}).  Given that at least one inequality (\ref{eqn:dominance}) holds with equality, let $J$ consist of the integers $j\in \{1,\hdots,r-1\}$ such that 
$$
\sum_{i=1}^j \lambda_i = \sum_{i=1}^j \mu_i.
$$
There are two possibilities:
\begin{enumerate}
\item[(A)] For all $j\in J$, $\lambda_{j+1}=0$.
\item[(B)] For some $j\in J$, $\lambda_{j+1}>0$.
\end{enumerate}

In case (A), take $r' = \min J$. Note that 
$\lambda_{r'+1}\ge \mu_{r'+1}$ 
by (\ref{eqn:dominance}) for $t=r'+1$, so $\lambda_k = \mu_k = 0$ for all $r'+1\le k\le r$. Then $(\lambda,\mu)$ comes from ${\sf Kostka}_{r'}$ by appending trailing $0$s. Furthermore, the inequalities (\ref{eqn:dominance}) hold strictly for all $1\le j<r'$, so $(\lambda,\mu) = (a^b,b^a)$, after scaling, by Lemma \ref{lem:easycase}. 

Now assume we are in case (B). Pick any $j\in J$ with $\lambda_{j+1}>0$. Therefore $\lambda_j>0$ as well. Set $\lambda_{\le j}$ to be the partition consisting just of $\lambda_1,\hdots,\lambda_j$. Let $\lambda_{>j}$ be the rest of $\lambda$. That is,
$$
\lambda=\lambda_{\le j}\concat\lambda_{>j}.
$$
Likewise, $\mu = \mu_{\le j}\concat \mu_{>j}$. Define rational partitions
\begin{align*}
\lambda^{(1)}&:=(1^j)\concat \left(\frac{1}{\lambda_j}\right) \lambda_{>j}\\
\lambda^{(2)}&:=(\lambda_{\le j}-(1^j))\concat \left(\frac{\lambda_j-1}{\lambda_j}\right) \lambda_{>j};
\end{align*}
note that $\lambda=\lambda^{(1)}+\lambda^{(2)}$. The sequence $\lambda^{(1)}$ is a rational partition since 
\[1\ge\frac{1}{\lambda_j}\lambda_{j+1}.\] 
Likewise $\lambda^{(2)}$ is nonincreasing since 
$$
\lambda_j-1\ge (\lambda_j-1)\cdot \frac{\lambda_{j+1}}{\lambda_j}.
$$
Furthermore, $\lambda^{(2)}$ has no negative entries since $\lambda_j-1\ge 0$ ($\lambda_j$ is an integer greater than $0$), so $\lambda^{(2)}$ is a rational partition. 

Observe that 
\begin{equation}
\label{eqn:Jan26jjj}
\mu_j\ge \lambda_j
\end{equation}
 since 
\begin{equation}
\label{eqn:Jan27yyy}
\sum_{i=1}^j \lambda_i = \sum_{i=1}^j \mu_i \text{\ \  and \ \ $\sum_{i=1}^{j-1} \lambda_i \ge \sum_{i=1}^{j-1} \mu_i$.}
\end{equation}
In particular $\mu_j>0$ since we already saw above that $\lambda_j>0$. Define
\begin{align*}
\mu^{(1)}&:=(1^j)\concat \left(\frac{1}{\lambda_j}\right) \mu_{>j}\\
\mu^{(2)}&:=(\mu_{\le j}-(1^j))\concat \left(\frac{\lambda_j-1}{\lambda_j}\right) \mu_{>j};
\end{align*}
once again $\mu^{(1)},\mu^{(2)}$ are rational partitions such that $\mu = \mu^{(1)}+\mu^{(2)}$. To see this for
$\mu^{(1)}$, note that
\begin{equation}
\label{eqn:Jan27jjj}
\lambda_j\geq \lambda_{j+1}\geq\mu_{j+1},
\end{equation} 
where the rightmost inequality holds by the equality in (\ref{eqn:Jan27yyy}). For $\mu^{(2)}$, we have
\[\mu_j-1\geq \lambda_j-1=\frac{\lambda_j-1}{\lambda_j}\lambda_j\geq \frac{\lambda_j-1}{\lambda_j}\mu_{j+1},\]
where the last inequality is by (\ref{eqn:Jan27jjj}).

\begin{claim}
We have $\lambda^{(1)}\ge_{\sf Dom} \mu^{(1)}$ and $\lambda^{(2)}\ge_{\sf Dom} \mu^{(2)}$. 
\end{claim}

\begin{proof}
We will show that $(\lambda^{(1)},\mu^{(1)})$ satisfy inequalities (\ref{eqn:dominance}). For any $t\le j$, the inequality (\ref{eqn:dominance}) for $t$ holds (trivially) with equality for $(\lambda^{(1)},\mu^{(1)})$. Therefore the inequalities (\ref{eqn:dominance}) for
$j<t<r$ are equivalent to the inequalities 
\begin{equation}
\label{eqn:truncdom}
\frac{1}{\lambda_j}\sum_{i=j+1}^{t} \lambda_i\ge \frac{1}{\lambda_j}\sum_{i=j+1}^{t} \mu_i,
\end{equation}
which already hold by the assumption 
\[\sum_{i=1}^j \lambda_i = \sum_{i=1}^j \mu_i\] 
and (\ref{eqn:dominance}).
An analogous argument shows $(\lambda^{(2)},\mu^{(2)})$ satisfy (\ref{eqn:dominance}). 
\end{proof}

Therefore we have written 
$$
(\lambda,\mu) = (\lambda^{(1)},\mu^{(1)})+(\lambda^{(2)},\mu^{(2)})
$$
as the sum of elements of ${\sf Kostka}_r$, so these summands must be parallel. This has several implications: apparently $\lambda_1 = \lambda_2 = \hdots = \lambda_j$ and $\mu_1 = \mu_2 = \hdots = \mu_j$. Furthermore, 
\[\lambda_1\ge \mu_1 \ge \mu_j \ge \lambda_j,\] 
where the last inequality is (\ref{eqn:Jan26jjj}). Thus,
\begin{equation}
\label{eqn:Jan26ggg}
\lambda_a=\lambda_b=\mu_a=\mu_b \text{\ for all $a,b\in [1,j]$.}
\end{equation}

\begin{claim} 
\label{claim:Jan26ttt}
$\lambda_{j+1}=\lambda_j$. 
\end{claim}
\begin{proof}
Assuming $\lambda_{j+1}<\lambda_j$, we also have 
\[\mu_{j+1}\le \lambda_{j+1}<\lambda_j =\mu_j,\]
where the leftmost inequality follows from (\ref{eqn:dominance}) and the fact that $j\in J$, and the equality holds by (\ref{eqn:Jan26ggg}).
Hence 
\[(\lambda,\mu)\pm \epsilon(1^j,1^j)\in {\sf Kostka}_r\] 
for all $0<\epsilon\leq 1$. Since $\lambda_{j+1}>0$, we know 
$(1^j,1^j)$ is not parallel to $(\lambda,\mu)$. This implies $(\lambda,\mu)$ is not extremal in ${\sf Kostka}_r$, a contradiction.
Hence $\lambda_{j+1}=\lambda_j$. 
\end{proof}

Recall the truncated partitions $\lambda_{>j} = (\lambda_{j+1}\ge \hdots \ge \lambda_r)$ and $\mu_{>j} = (\mu_{j+1}\ge \hdots \ge \mu_r)$.  Since $\sum_{i=1}^j \lambda_i = \sum_{i=1}^j \mu_i$, we know that $|\lambda_{>j}| = |\mu_{>j}|$. Moreover, $\lambda_{>j} \ge_{\sf Dom} \mu_{>j}$ by (\ref{eqn:truncdom}). So $(\lambda_{>j}, \mu_{>j})\in {\sf Kostka}_{r-j}$. 

\begin{claim}
$(\lambda_{>j},\mu_{>j})$ lies on an extremal ray of ${\sf Kostka}_{r-j}$. 
\end{claim}

\begin{proof}
If not, we can decompose $\lambda_{>j} = \bar \lambda^{(1)}+ \bar \lambda^{(2)}$ and $\mu_{>j} = \bar \mu^{(1)}+ \bar \mu^{(2)}$ where $(\bar \lambda^{(1)},\bar \mu^{(1)})$ and $ (\bar\lambda^{(2)}, \bar \mu^{(2)})$ are nonparallel elements of ${\sf Kostka}_{r-j}$. 
Define the following concatenations of real partitions: 
\begin{align*}
\lambda^{(1)}&:=\left(\frac{\bar\lambda^{(1)}_1}{\lambda_1}\right)\lambda_{\le j}\concat \bar\lambda^{(1)}\\
\lambda^{(2)}&:=\left(\frac{\bar\lambda^{(2)}_1}{\lambda_1}\right)\lambda_{\le j}\concat \bar\lambda^{(2)} 
\end{align*}
and 
\begin{align*}
\mu^{(1)}&:=\left(\frac{\bar\lambda^{(1)}_1}{\lambda_1}\right)\mu_{\le j}\concat \bar\mu^{(1)} \\
\mu^{(2)}&:=\left(\frac{\bar\lambda^{(2)}_1}{\lambda_1}\right)\mu_{\le j}\concat \bar\mu^{(2)} 
\end{align*}

\begin{subclaim}
Each of $\lambda^{(1)}$, $\lambda^{(2)}$, $\mu^{(1)}$, $\mu^{(2)}$ is a real partition with at most $r$ nonzero parts. 
\end{subclaim}

\begin{proof}
In each concatenation, the two pieces are separately real partitions. 
We need only verify that at each ``$\concat$'' the last element of the first piece is weakly bigger than the first element of the second piece. 
\begin{itemize}
\item ($\lambda^{(1)}$): Using (\ref{eqn:Jan26ggg}), $\left(\frac{\bar\lambda^{(1)}_1}{\lambda_1}\right) \lambda_j = \left(\frac{\bar\lambda^{(1)}_1}{\lambda_1}\right) \lambda_1 = \bar\lambda^{(1)}_1$, so the concatenation is good and $\lambda^{(1)}$ is a real partition. 
\item ($\lambda^{(2)}$): Use the same proof as for $\lambda^{(1)}$, switching upper indices everywhere. 
\item ($\mu^{(1)}$): Once again by (\ref{eqn:Jan26ggg}), we have $\mu_j = \lambda_1$. Since $\bar \lambda^{(1)} \ge_{\sf Dom}\bar\mu^{(1)}$, $\bar\lambda^{(1)}_1 \ge \bar\mu^{(1)}_1$. Putting these together, we get  
$\left(\frac{\bar\lambda^{(1)}_1}{\lambda_1}\right) \mu_j = \left(\frac{\bar\lambda^{(1)}_1}{\lambda_1}\right) \lambda_1 = \bar\lambda^{(1)}_1\ge \bar\mu^{(1)}_1$, as needed. 
\item ($\mu^{(2)}$): Use the same proof as for $\mu^{(1)}$. \qedhere
\end{itemize}
\end{proof}

Since the first $j$ entries of $\lambda^{(1)}$ and $\mu^{(1)}$ are the same, $\bar\lambda^{(1)}\ge_{\sf Dom} \bar\mu^{(1)} \implies \lambda^{(1)}\ge_{\sf Dom} \mu^{(1)}$. Likewise $\lambda^{(2)}\ge_{\sf Dom} \mu^{(2)}$. Observe that 
$$
\frac{\bar\lambda^{(1)}_1}{\lambda_1}+\frac{\bar\lambda^{(2)}_1}{\lambda_1} = \frac{(\lambda_{>j})_1}{\lambda_1} = \frac{\lambda_{j+1}}{\lambda_1} = 1, 
$$
where the last equality comes from Claim \ref{claim:Jan26ttt} and (\ref{eqn:Jan26ggg}). This implies that $\lambda = \lambda^{(1)}+\lambda^{(2)}$ and $\mu = \mu^{(1)}+\mu^{(2)}$. Since $(\bar\lambda^{(1)},\bar\mu^{(1)})$ and $(\bar\lambda^{(2)},\bar\mu^{(2)})$ are nonparallel, $(\lambda^{(1)}, \mu^{(1)})$ and $(\lambda^{(2)},\mu^{(2)})$ are also nonparallel. Therefore we have obtained a nonparallel decomposition of $(\lambda,\mu)$ inside ${\sf Kostka}_r$, which contradicts our assumption that $(\lambda,\mu)$ is an extremal ray. 
\end{proof} 

By the induction hypothesis, there exist integers $a\ge b>0$ and $k\ge 0$ such that $(\lambda_{>j}, \mu_{>j})\in {\sf Kostka}_{r-j}$ is parallel to $(a^{k+b}, a^{k}b^a)$. 
Hence there exists a positive real number $q$ such that 
\[(\lambda_{>j}, \mu_{>j}) = q(a^{k+b}, a^{k}b^a).\] 
This makes $\lambda_{j+1}=qa$. By Claim \ref{claim:Jan26ttt} and (\ref{eqn:Jan26ggg}),  
$\lambda = q(a^{j+k+b}) \text{  and  } \mu = q(a^{j+k}b^a)
$, as desired (let $\ell=j+k$). This completes the proof of Proposition~\ref{prop:thefinalprop}.
\end{proof}

\section{Conjectures}\label{sec:5}

\subsection{Generalized Catalan sequences}
Let 
${\vec x}=(x_1,\ldots,x_t)$ be a sequence of nonzero integers. We say ${\vec x}$ is \emph{generalized Catalan} if 
\begin{equation}
\label{eqn:gencat1}
\sum_{i=1}^{t} x_i=0
\end{equation}
 and
 \begin{equation}
 \label{eqn:gencat2}
\sum_{i=1}^q x_i\geq 0 \text{\ for all $1\leq q\leq t$.}
\end{equation} 
Define ${\vec x}$ to be \emph{reducible} if there is a Catalan sublist ${\vec x}^{\circ}=(x_{i_1},x_{i_2},\ldots, x_{i_a})$ such that the
complementary sublist ${\vec x}^{\bullet}$ is also Catalan.

\begin{example}
\label{exa:Jan26xyx}
${\vec x}=(\underline{3},2,1,-2,1,-2,\underline{-1},\underline{-1},2,-1,2,1,-2,-1,-1,\underline{-1})$ is a Catalan list. The underlined
elements and non-underlined elements separately define two Catalan sublists ${\vec x}^{\circ}=(3,-1,-1,-1)$ and
${\vec x}^{\bullet}=(2,1,-2,1,-2,2,-1,2,1,-2,-1,-1)$ that decompose ${\vec x}$.
\end{example}

A maximal consecutive subsequence of ${\vec x}$ consisting of integers of the same sign is called a \emph{run}. By condition (\ref{eqn:gencat1}) there are an even number $2y$ of runs. Let $a_k>0$ be the maximum (in absolute value) of any $x_i$ in run $k$. Define 
\[{\sf cost}({\vec x})=\sum_{k=1}^{2y} a_k \text{\ and ${\sf width}({\vec x})=q$.}\]

\begin{example}
Continuing Example~\ref{exa:Jan26xyx}, $y=4$, ${\sf cost}({\vec x})=3+2+1+2+2+1+2+2=15$ and ${\sf width}({\vec x})=16$.
\end{example}

\begin{conjecture}
\label{conj:cost-width}
If ${\sf cost}({\vec x})<{\sf width}({\vec x})$ then ${\vec x}$ is reducible.\footnote{After the original version of this  preprint was posted to the {\sf arXiv}, J.~Kim \cite[Theorem~1.1]{Kim} proved this conjecture.}
\end{conjecture}

\begin{definition}
Let $(\lambda,\mu)\in {\sf Kostka}_{r}^{\mathbb Z}$. A decomposition 
$(\lambda,\mu)=(\lambda^{\bullet},\mu^{\bullet})+(\lambda^{\circ},\mu^{\circ})$
is a \emph{commonly reducible}
if one can choose common columns of $\lambda$ and $\mu$ to give $\lambda^{\bullet}$ and $\mu^{\bullet}$.
\end{definition}

This asserts a strengthening of Theorem~\ref{thm:bound}:

\begin{conjecture}
\label{conj:commonly}
If $\lambda_1>r$ then $(\lambda,\mu)$ is commonly reducible.
\end{conjecture}

\begin{example}
Let $r=4$ and consider the decomposition
\[(\lambda,\mu)= \ktableau{ \  & \ & \ & \ & \ \\ \ \\ \ \\ }, \ktableau{ \ & \ \\ \ & \ \\ \ & \ \\ \ }=
\ktableau{\ & \ \\ \ \\ \ }, \ktableau{\ \\ \ \\ \ \\ \ } + \ktableau{\ & \ & \ }, \ktableau{\ \\ \ \\ \ } = (\lambda^{\bullet},\mu^{\bullet})+(\lambda^{\circ},\mu^{\circ}).
\]
This demonstrates the common reducibility of $(\lambda,\mu)$. Here we
interpret $(\lambda^{\bullet},\mu^{\bullet})$ as obtained from $(\lambda,\mu)$ by using the (possibly empty) columns $1,5$
and $(\lambda^{\circ},\mu^{\circ})$ as obtained from $(\lambda,\mu)$ by using the complementary columns. Notice that
\[(\lambda,\mu)= \ktableau{ \  & \ & \ & \ & \ \\ \ \\ \ \\ }, \ktableau{ \ & \ \\ \ & \ \\ \ & \ \\ \ }=
\ktableau{\ \\ \ \\ \ }, \ktableau{\ \\ \ \\ \  } + \ktableau{\ & \ & \ & \ }, \ktableau{\ \\ \ \\ \ \\ \ }
\]
is another decomposition that does not exhibit the common reducibility.

The hypothesis of Conjecture~\ref{conj:commonly} cannot be dispensed with. For instance,
\[(\lambda,\mu)= \ktableau{ \ & \ & \ \\ \ & \  & \ \\ \ }, \ktableau{ \ & \ \\ \ & \ \\ \ & \ \\ \ }=
\ktableau{\ \\ \ \\ \ }, \ktableau{\ \\ \ \\ \  } + \ktableau{\ & \ \\ \ & \  }, \ktableau{\ \\ \ \\ \ \\ \ }
\]
is not in the Hilbert basis of ${\sf Kostka}_4$, but is not commonly reducible.
\end{example}

\begin{proposition}
Conjecture~\ref{conj:cost-width} $\implies$ Conjecture~\ref{conj:commonly}.
\end{proposition}
\begin{proof}
Assume $\lambda\geq_{\sf Dom} \mu$. Define a sequence $\vec{x}$ of length $\lambda_1$ by
$x_j:=\mu_j'-\lambda_j'$. 
Since $|\lambda|=|\mu|$, (\ref{eqn:gencat1}) holds. If $x_j=0$ for any $j$, then
$(\lambda,\mu)$ is trivially commonly reducible, so we may assume otherwise. Now notice that
$\lambda'\leq_{\sf Dom} \mu'$ (equivalently, $\mu\leq_{\sf Dom}\lambda$) is clearly equivalent to (\ref{eqn:gencat2}). Furthermore, from the definitions 
one checks ${\sf cost}({\vec x})\leq \ell(\mu)$. Hence, under the hypothesis of Conjecture~\ref{conj:commonly},
\[{\sf cost}({\vec x})\leq \ell(\mu)<r<\lambda_1={\sf width}({\vec x}).\] 
Therefore Conjecture~\ref{conj:cost-width}'s conclusion is that
${\vec x}$ is decomposable into generalized Catalan sequences 
${\vec x}^{\bullet}$ and ${\vec x}^{\circ}$, which correspond to, say, columns
$C$ and $[\lambda_1]-C$. Then define $\lambda^\bullet$ and $\lambda^{\circ}$ to be columns $C$ and $[\lambda_1]-C$
of $\lambda$, respectively. Similarly define $\mu^{\bullet}$ and $\mu^{\circ}$. Since ${\vec x}^{\bullet}$ is Catalan, by the equivalence
stated earlier in this proof, $\lambda^{\bullet}\geq_{\sf Dom} \mu^{\bullet}$. Similarly $\lambda^{\circ}\geq_{\sf Dom} \mu^{\circ}$.
Thus 
$(\lambda,\mu)=(\lambda^{\bullet},\mu^{\bullet})+
(\lambda^{\circ},\mu^{\circ})$ 
witnesses the common reducibility.
\end{proof}

Using an argument similar to the one for \cite[Theorem 6.1]{Sturmfels} we have a proof of Conjecture~\ref{conj:commonly} in the
case $y=1$.

Let us also remark that one can study the decision problem of whether a generalized Catalan sequence ${\vec x}$ is reducible. A modification of the argument
for Theorem~\ref{thm:complexity} shows that this problem is also ${\sf NP}$-complete.

%

\subsection{The Littlewood-Richardson cone}\label{sec:5.2}
For partitions 
\[\lambda=(\lambda_1,\lambda_2,\ldots,\lambda_r), \mu=(\mu_1,\mu_2,\ldots,\mu_r), \nu=(\nu_1,\nu_2,\ldots,\nu_r),\]
let $c_{\lambda,\mu}^{\nu}$ be the \emph{Littlewood-Richardson coefficient}. Combinatorially, $c_{\lambda,\mu}^{\nu}$ counts the
number of semistandard tableaux $T$ of skew shape $\nu/\lambda$ of content $\mu$ such that the right to left, top to bottom, row reading word is a ballot sequence \cite{ECII}. Define
\[{\sf LR}_r^{\mathbb Z}=\{(\lambda,\mu,\nu): c_{\lambda,\mu}^{\nu}>0\}.\]
Like ${\sf Kostka}_r^{\mathbb Z}$, ${\sf LR}_r^{\mathbb Z}$ is a finitely generated semigroup; more precisely,
${\sf LR}_r^{\mathbb Z}$ are the lattice points of a pointed polyhedral cone ${\sf LR}_r$ defined by the celebrated \emph{Horn inequalities}; we refer to the survey \cite{Fulton} and the references therein. 
A.~Zelevinsky \cite{Zelevinsky} raised the question of studying the Hilbert basis of
${\sf LR}_r^{\mathbb Z}$. Analogous to Theorem~\ref{thm:complexity}, we conjecture that the decision problem of deciding if $(\lambda,\mu,\nu)$
is a Hilbert basis element is also {\sf NP}-complete. 

In a previous preprint version of this work, we conjectured:
\begin{conjecture}
\label{conj:LRthing}
If $(\lambda,\mu,\nu)$ is in the Hilbert basis of ${\sf LR}_r^{\mathbb Z}$ then $\nu_1\leq r$.
\end{conjecture}

However, we now know this to be false, due to an explicit family of counterexamples that were constructed
by the second and third coauthors; see Appendix B.

\appendix

\section{${\sf NP}$-completeness and the Hilbert basis}\label{sec:2}

The \emph{Subset Sum problem} ({\tt SubsetSum}) takes as input positive integers $a_1,\ldots,a_d$ and 
$b$. The output is ``yes'' if there exists $S\subseteq \{1,2,\ldots,d\}$ such that
\begin{equation}
\label{eqn:thesum}
\sum_{i\in S} a_i = b.
\end{equation}
This problem is well-known to be {\sf NP}-complete. 

Clearly ${\tt KostkaHilbert} \in {\sf NP}$: given a proposed decomposition (\ref{eqn:decomposition}) it takes polynomial time to check it.
Thus, to prove that {\tt KostkaHilbert} is
{\sf NP}-{\sf complete} we give a polynomial-time reduction of {\tt SubsetSum} to it.

Given the input $a_1,\ldots a_d, b$ to {\tt SubsetSum}, let 
\[A=\sum_{i=1}^d a_i.\]  
After sorting (which takes $O(d\log d)$-time), we may assume
$a_1\geq a_2\geq \cdots \geq a_d$. We may also assume $b\leq A$, since otherwise
{\tt SubsetSum} is trivial; this condition can be checked in $O(d)$-time. Now define 
$\overline{\lambda}$ to have columns 
$A+1,a_1,a_2,\ldots,a_d$. Set $\overline{\mu}$ to have columns $A+1+(A-b)$ and $b$.

\begin{lemma}
\label{lemma:domcheck}
$\overline{\lambda}\geq_{\sf Dom} \overline{\mu}$.
\end{lemma}
\begin{proof}
By construction, 
$|\overline{\lambda}|=(A+1)+A=A+1+(A-b)+b=|\overline{\mu}|$.
Since $\leq_{\sf Dom}$ is an anti-automorphism with respect to conjugation \cite[Section~7.2]{ECII},
it is equivalent to show that $\overline{\mu}'\geq_{\sf Dom} \overline{\lambda}'$. Since 
$\overline{\mu}'$ only has two rows, this inequality is immediate since $\overline{\mu}_1'=A+1+(A-b)\geq A+1$.
\end{proof}

In view of Lemma~\ref{lemma:domcheck}, $(\overline{\lambda},\overline{\mu})$ is not a Hilbert basis
element if it has a nontrivial decomposition (\ref{eqn:decomposition}). This next observation is clear:

\begin{lemma}
\label{lemma:gidoncol}
$(\lambda,\mu)$ is not a Hilbert basis element  if and only if 
both of the following hold:
\begin{itemize}
\item[(I)]
the columns
of $\lambda^{\bullet}$ (resp.~$\mu^{\bullet}$) and $\lambda^{\circ}$ (resp.~$\mu^{\circ}$) have multiset union equal to the columns of
$\lambda$ (resp.~$\mu$); and
\item[(II)] $\lambda^{\bullet}\geq_{\sf Dom} \mu^{\bullet}$ and
$\lambda^{\circ}\geq_{\sf Dom} \mu^{\circ}$.
\end{itemize}
\end{lemma}

To complete the proof of Theorem~\ref{thm:complexity}, it remains to show:

\begin{proposition}
$(\overline{\lambda},\overline{\mu})$ is not a Hilbert basis if and only if {\tt SubsetSum} returns {\tt Yes}.
\end{proposition}
\begin{proof}
($\Rightarrow$) Suppose we have a nontrivial decomposition (\ref{eqn:decomposition}) of
$(\overline{\lambda},\overline{\mu})$. By Lemma~\ref{lemma:gidoncol}(I), since $\mu$ has precisely two columns (of length $A+1+(A-b)$ and $b$), we may
assume $\overline{\mu}^{\bullet}=(1^b)$. For the same reason, the columns of 
$\overline{\lambda}^{\bullet}$
must be a subset $S'$ of the columns of ${\overline\lambda}$. However, since the first column of $\overline\lambda$ is of length $A+1>b$, that column cannot be in $S'$. Hence
we may take $S$ to be the columns of $\lambda$ used in $S'$ (after reindexing). By construction, these columns are of length $a_1,\ldots,a_d$ and since 
\[\sum_{i\in S} a_i=|\overline{\lambda^{\bullet}}|=|\overline{\mu}^{\bullet}|=|(1^b)|=b,\]
{\tt SubsetSum} returns {\tt Yes}, as desired. 

\noindent
($\Leftarrow$) Suppose $S\subseteq \{1,2,\ldots,d\}$ satisfies (\ref{eqn:thesum}).
Let $\overline{\lambda}^{\bullet}$ be the partition formed by taking the columns of 
$\lambda$ indexed by $S$; $\overline{\lambda}^{\circ}$ is the partition consisting of the
remaining columns of $\lambda$ together with the first column (of length $A+1+(A-b)$)
of $\overline{\lambda}$.  Now, let $\overline{\mu}^{\bullet}=(1^b)$ and $\overline{\mu}^{\circ}=1^{A+1+(A-b)}$. Clearly, (\ref{eqn:decomposition}) holds. Moreover, since 
$|\overline{\lambda}^{\bullet}|=b=|\overline{\mu}^{\bullet}|$ and $\overline{\mu}^{\bullet}$
is a column partition, $\overline{\mu}^{\bullet}\leq_{\sf Dom} \overline{\lambda}^{\bullet}$
is immediate. Similarly, $\overline{\mu}^{\circ}\leq_{\sf Dom} \overline{\lambda}^{\circ}$ is true. Hence $(\overline{\lambda},\overline{\mu})$ is reducible and therefore not a Hilbert basis element.
\end{proof}

Theorem~\ref{thm:complexity} can be compared to two related facts. By (\ref{eqn:dominance}) one has a
polynomial-time algorithm (using the same input encoding) for deciding if $(\lambda,\mu)\in {\sf Kostka}_r^{\mathbb Z}$. 
In contrast, \emph{counting} $K_{\lambda,\mu}$ is $\#{\sf P}$-complete (using an encoding of
partitions by rows) \cite{Narayanan}.

\section{Counterexample to Conjecture~\ref{conj:LRthing}, relationship to P.~Belkale's conformal block question (by J.~Kiers and G.~Orelowitz)}
\label{appendix:B}

The following result gives a counterexample to Conjecture~\ref{conj:LRthing}. In fact, it also shows that one 
cannot replace the bound of $r$ by any linear function of $r$.

\begin{theorem}
\label{thm:counterex}
There does not exist any linear function $f(r)$ such that for all $(\lambda,\mu,\nu)$ in the Hilbert basis of ${\sf LR}_r^{\mathbb Z}$, $\nu_1\leq f(r)$. 
\end{theorem}

\begin{proof}
The extremal rays of the Littlewood-Richardson cone are inductively
determined  by P.~Belkale \cite{Belkale}. 
For each $r\equiv 2 \pmod{3}$, we will use Belkale's formulas (as recalled in \cite{K:21}) to show that there exists a (``type I'') extremal ray for ${\sf LR}_r$, $(\lambda,\mu,\nu)$, such that $\nu_1\in \Omega(r^2)$. 

Fix $k\ge 2$, and let $r = 3k-1$.  Let 
\[I=\{1,2,3,\dots, k-2,k-1,2k\}, J = \{1,4,7,\dots, 3k-2\}, \text{ \ and $K = \{2,5,8,\dots, 3k-1\}$.}\]  
Then we have that 
\[\tau(I) = (k), \tau(J) = (2k-2,2k-4,\dots, 2), \text{ \ and $\tau(K) = (2k-1,2k-3,\dots, 3,1)$.}\]  
Clearly, $\tau(K)\setminus \tau(J)$ is a horizontal strip of size $k$, so by Pieri's rule $c^{\tau(K)}_{\tau(I),\tau(J)} = 1$, so $(I,J,K)$ gives a face of the Littlewood-Richardson cone, thus satisfying the initial condition of the extremal ray generation algorithm. 

Now, we choose to increase the subset $I$ by replacing $k-1$ with $k$. Let $I'=\{1,2,3,\dots, k-2,k,2k\}$, $J' = J$, and $K' = K$.  Now $\tau(I') = (k,1)$.  We will construct the corresponding $\lambda,\mu,\nu$ separately according to the algorithm:

\noindent \emph{Calculating $\lambda$:}  There are only two elements of $I'$ that can be reduced by $1$: the $k$ and the $2k$.  If we subtract $1$ from the $k$, then $I'' = I$, so $\lambda_{k-1} - \lambda_k = c^{\tau(K')}_{\tau(I''),\tau(J')} = c^{\tau(K)}_{\tau(I),\tau(J)}=1$.

On the other hand, if we subtract $1$ from the $2k$, then $\tau(I'') = \tau(\{1,2,3,\dots, k-2,k,2k-1\}) = (k-1,1)$.  According to the Littlewood-Richardson rule, this means that $c^{\tau(K')}_{\tau(I''),\tau(J')}$ is the number of ballot tableaux of shape $(2k-1,2k-3,\dots, 3,1)\setminus (2k-2,2k-4,\dots, 4,2)$ filled with $k-1$ $1$'s and one $2$.  As long as the $2$ is placed in any row other than the first, the filling will be ballot, and since $(2k-1,2k-3,\dots, 3,1)\setminus (2k-2,2k-4,\dots, 4,2)$ has $k$ rows, $\lambda_{2k-1} - \lambda_{2k} = c^{\tau(K')}_{\tau(I''),\tau(J')}=k-1$.

For all other $i\in [r]$, either $i\in I'$ or both $i,i+1\not\in I'$, so $\lambda_i-\lambda_{i+1}=0$ for all $i\not= k-1,2k-1$.  Therefore, $\lambda = ((k)^{k-1},(k-1)^k)$.

\noindent \emph{Calculating $\mu$:} For each $i\in [k]\setminus \{1\}$, $3i-2\in J'$, but $3i-3\not\in J'$, so we can subtract $1$ from $3i-2$ to let $J'' = \{1,4,\dots,3i-5,3i-3,3i+1,3i+4,\dots 3k-2\}$.  According to the Littlewood-Richardson rule, this means that $c^{\tau(K')}_{\tau(I'),\tau(J'')}$ is the number of ballot tableaux of shape $(2k-1,2k-3,\dots, 3,1)\setminus (2k-2,2k-4,\dots, 2i, 2i-3,2i-4,2i-6, \dots ,2)$ filled with $k$ $1$'s and one $2$.  If the two is in the rightmost box of any row other than the first, the filling will be ballot, and otherwise it will not be, and since $(2k-1,2k-3,\dots, 3,1)\setminus (2k-2,2k-4,\dots, 4,2)$ has $k$ rows, $\mu_{3i-3} - \mu_{3i-2} = c^{\tau(K')}_{\tau(I'),\tau(J'')}=k-1$. 

For all other $i\in [r]$, either $i\in J'$ or both $i,i+1\not\in J'$, so $\mu_i-\mu_{i+1}=0$ for all $i\not= 3j-2$ for any $j$.  Therefore, $\mu = (((k-1)(k-1))^3,((k-2)(k-1))^3,\dots, (k-1)^3)$.

\noindent \emph{Calculating $\nu$:} For each $i\in [k-1]$, $3i-1\in K'$, but $3i\not\in K'$, so we can add $1$ to $3i-1$ to let $K'' = \{2,5,\dots,3i-4,3i,3i+2,3i+5,\dots 3k-1\}$.  According to the Littlewood-Richardson rule, this means that $c^{\tau(K'')}_{\tau(I'),\tau(J')}$ is the number of ballot tableaux of shape $(2k-1,2k-3,\dots, 2i+1, 2i,2i-3,2i-5, \dots ,3,1)\setminus (2k-2,2k-4, \dots ,2)$ filled with $k$ $1$'s and one $2$.  If the two is in the rightmost box of any row other than the first, the filling will be ballot, and otherwise it will not be, and since $(2k-1,2k-3,\dots, 2i+1, 2i,2i-3,2i-5, \dots ,3,1)\setminus (2k-2,2k-4, \dots ,2)$ has $k$ rows, $\nu_{3i-1} - \nu_{3i} = c^{\tau(K'')}_{\tau(I'),\tau(J')}=k-1$. 

For all other $i\in [r]$, either $i\in K'$ or both $i,i+1\not\in K'$, so $\nu_i-\nu_{i+1}=0$ for all $i\not= 3j-1$ for any $j$.

By \cite[Algorithm 3.2(4)]{K:21}, we can find $\nu_r$, and thus finish determining $\nu$, in one of two ways. One can simply use $K'' = \{2,5,\dots,3k-4,3k\}$ and calculate $\nu_r = c^{\tau(K'')}_{\tau(I'),\tau(J')}=k-1$ (same reasoning as above). Or, one can solve the equation $|\lambda|+|\mu| = |\nu|$: we know that $\nu = (((k-1)(k-1))^2,((k-2)(k-1))^3,((k-3)(k-1))^3,\dots, (k-1)^3)) + (\nu_r^r)$, where $\nu_r$ is to be determined.  We see that $|\lambda| = 2k(k-1)$, $|\mu| = \frac{3}{2}k(k-1)^2$, and $|\nu| = (3k-1)\nu_r + \frac{3}{2}k(k-1)^2 - (k-1)^2$.  We get that $\nu_r = k-1$ once again. So $\nu = ((k(k-1))^2,((k-1)(k-1))^3,((k-2)(k-1))^3,\dots, (k-1)^3)$.

Since $n=3k-1$, $k = \frac{r+1}{3}$, and so we have constructed an extremal ray of ${\sf LR}_r$ such that $\nu_1 = \frac{r+1}{3}(\frac{r+1}{3}-1) = \Omega(r^2)$.  Since $\lambda_1 = k$ and $\lambda_{k} = k-1$ are relatively prime, $(\lambda,\mu,\nu)$ is the first lattice point along its one dimensional face.  Therefore, $(\lambda,\mu,\nu)$ is in the Hilbert basis for ${\sf LR}_r$, completing the proof.
\end{proof}

In fact, Theorem \ref{thm:counterex} has an application to a more general setting. 
Analogous to the Littlewood-Richardson coefficients  $c_{\lambda,\mu}^\nu$ are the spaces of conformal blocks $\mathbb{V}(\lambda,\mu,\nu,\ell)$ and their dimensions $c_{\lambda,\mu}^\nu(\ell)$; here $\lambda,\mu,\nu$ are still partitions of length at most $r$, and $\ell$ is a natural number. These always satisfy 
\begin{align}\label{eq:cblocks}
c_{\lambda,\mu}^\nu(\ell) &\le c_{\lambda,\mu}^\nu;\\\nonumber
\lim_{\ell\to \infty} c_{\lambda,\mu}^\nu(\ell) &= c_{\lambda,\mu}^\nu.
\end{align}
Akin to the cone ${\sf LR}_r$, there is a pointed rational cone ${\sf CB}_r$ generated by all $(\lambda,\mu,\nu,\ell)$ such that $c_{\lambda,\mu}^\nu(\ell)\ne 0$. By (\ref{eq:cblocks}), there is a natural surjective map ${\sf CB}_r\twoheadrightarrow {\sf LR}_r$ by forgetting the $\ell$-coordinate. 

The extremal rays of ${\sf CB}_r$ have been studied by Belkale in \cite{Belk2}, and the rays formulas of \cite{Belkale} and \cite{Belk2} are compatible in that every type I extremal ray of ${\sf LR}_r$ has a lift to a type I extremal ray of ${\sf CB}_r$. 

A trivial property of the coefficients $c_{\lambda,\mu}^\nu(\ell)$ is that they are $0$ unless $\max\{\lambda_1-\lambda_r, \mu_1-\mu_r, \nu_1-\nu_r\}\le \ell$. So in the proof of Theorem \ref{thm:counterex}, the type I ray $(\lambda,\mu,\nu)$ of ${\sf LR}_{3k-1}$ lifts to a type I ray $(\lambda,\mu,\nu,\ell)$ of ${\sf CB}_{3k-1}$ for which $\ell \ge (k-1)^2$ (in fact, $\ell = (k-1)^2+1$).  This violates the original speculation that $\ell<r$ for the type I extremal rays put forth in the preprint \cite[Remark 1.6.5]{Belk2}.

\section*{Acknowledgements}
We thank Allen Knutson, David Speyer and John Stembridge for helpful communications. AY was partially supported by a Simons Collaboration Grant. SG, GO and AY were supported  by NSF RTG DMS 1937241. SG was partially supported by the
National Science Foundation Graduate Research Fellowship Program under Grant No. DGE--1746047.

\end{document}